\documentclass{article}

\usepackage[numbers]{natbib}

\usepackage[utf8]{inputenc} 
\usepackage[T2A, T1]{fontenc}    
\usepackage{amsthm}
\usepackage{amsmath}
\usepackage{amssymb}
\usepackage{hyperref}       
\usepackage{url}            
\usepackage{booktabs}       
\usepackage{amsfonts}       
\usepackage{mathtools}      
\usepackage{nicefrac}       
\usepackage{microtype}      
\usepackage{subfiles}       
\usepackage{bbold}          
\usepackage{xcolor}         
\usepackage{bm}             
\usepackage{subcaption}     
\usepackage{caption}
\usepackage{enumitem}
\usepackage[margin=3cm]{geometry}
\usepackage{framed}
\usepackage{titlesec}
\usepackage{fullpage}
\usepackage{authblk}
\usepackage[ruled, lined]{algorithm2e}

 \mathtoolsset{showonlyrefs}

\definecolor{myred}{HTML}{880000}
\definecolor{mygreen}{HTML}{008800}
\definecolor{myblue}{HTML}{000088}
\definecolor{linkblue}{HTML}{0000BB}
\hypersetup{
  colorlinks,
  linkcolor={myred},
  citecolor={myblue},
  urlcolor={linkblue}
}

\newtheorem{theorem}{Theorem}
\newtheorem{proposition}{Proposition}
\newtheorem{lemma}{Lemma}

\theoremstyle{definition}
\newtheorem{definition}{Definition}

\theoremstyle{remark}

\newtheorem{setup}{Problem Setting}

\usepackage{macros}

\title{Computational Guarantees for Doubly Entropic Wasserstein Barycenters via
Damped Sinkhorn Iterations}

\author{L\'{e}naïc Chizat and Tomas Va\v{s}kevi\v{c}ius}
\affil{EPFL, Institute of Mathematics, Switzerland}

\begin{document}

\maketitle

\begin{abstract}
We study the computation of doubly regularized Wasserstein barycenters, a
recently introduced family of entropic barycenters governed by inner and outer
regularization strengths. Previous research has demonstrated that various
regularization parameter choices unify several notions of entropy-penalized
barycenters while also revealing new ones, including a special case of debiased
barycenters.  In this paper, we propose and analyze an algorithm for computing
doubly regularized Wasserstein barycenters. Our procedure builds on damped
Sinkhorn iterations followed by exact maximization/minimization steps and
guarantees convergence for any choice of regularization parameters. An inexact
variant of our algorithm, implementable using approximate Monte Carlo sampling,
offers the first non-asymptotic convergence guarantees for approximating
Wasserstein barycenters between discrete point clouds in the
free-support/grid-free setting.
\end{abstract}


\section{Introduction}
\label{sec:introduction}

The Wasserstein distance between two probability distributions measures the
least amount of effort needed to reconfigure one measure into the other.
Unlike other notions of distances based solely on the numerical values taken by
the distribution functions (e.g., the Kullback-Leibler divergence), the
Wasserstein distance incorporates an additional layer of complexity by
considering pairwise distances between distinct points, measured by some
predetermined cost function.  As a result, the Wasserstein distances can be seen
to lift the geometry of the underlying space where the probability
measures are defined to the space of the probability measures itself.
This allows for a more thorough and geometrically nuanced understanding of the
relationships between different probability measures, which
proved to be a versatile tool of increasing importance in a broad spectrum of
areas.

Given a collection of probability measures and an associated set of positive weights that sum
to one, the corresponding Wasserstein barycenter minimizes the  weighted
sum of Wasserstein distances to the given measures.
In the special case of two measures and the squared Euclidean cost function,
Wasserstein barycenters concide with the notion of McCann's displacement
interpolation introduced in the seminal paper  \cite{mccann1997convexity}.
The general case, encompassing an arbitrary number of measures, was first
studied by \citet*{agueh2011barycenters}, where they also demonstrated a close
link between Wasserstein barycenters and the multi-marginal optimal transport
problem \cite{gangbo1998optimal}.
Recent years have witnessed an increasing number of applications of
Wasserstein barycenters across various scientific disciplines.
See, for instance, the following sample of works in economics
\cite{chiappori2010hedonic,carlier2010matching}, statistics
\cite{bernton2019parameter}, image processing \cite{rabin2012wasserstein},
and machine learning \cite{courty2017joint}, among other areas.
For further background and references we point the interested reader to the introductory surveys
\cite{peyre2019computational, panaretos2019statistical}
and the textbooks \cite{villani2003topics,
villani2009optimal, santambrogio2015optimal,
panaretos2020invitation, figalli2021invitation}.

Despite their compelling theoretical characteristics, the computation of
Wasserstein barycenters poses significant computational challenges, particularly
in large-scale applications.
While Wasserstein barycenters can be computed in polynomial time
for fixed dimensions \cite{altschuler2021wasserstein},
the approximation of Wasserstein barycenters is known to be NP--hard
\cite{altschuler2022wasserstein}.
Currently employed methods for approximating Wasserstein barycenters are
predominantly based on space discretizations. Unfortunately,
such strategies are only computationally practical for problems of relatively
modest scale. Although there are a handful of grid-free techniques available for
approximating Wasserstein barycenters (e.g.,
\cite{cohen2020estimating, korotin2021continuous, daaloul2021sampling, lindheim2023simple}),
we are not aware of any existing methods that provide bounds on computational
complexity.
One contribution of the present paper is to introduce a method that in some
regimes can provably approximate Wasserstein barycenters without relying on
space discretizations, but instead employing approximate Monte Carlo sampling.

More broadly, the difficulties associated with computation of the
optimal transport cost has prompted the exploration of
computationally efficient alternatives, leading to the consideration of
regularized Wasserstein distances.
Among these, the entropic penalty has
emerged as one of the most successful in applications.
The practical success of entropic penalization can be attributed to
Sinkhorn's algorithm \cite{sinkhorn1967diagonal},
which enables efficient and highly parallelizable
computation, an algorithm that gained substantial traction in the machine
learning community following the work of \citet*{cuturi2013sinkhorn}.
It is worth noting that entropic Wasserstein distances are of intrinsic
interest, beyond their approximation capabilities. Indeed, they hold a rich
historical connection to the Schr\"{o}dinger bridge problem
\cite{schrodinger1932theorie, wilson1969use, erlander1990gravity}, as
highlighted in the recent surveys \cite{leonard2013survey, chen2021stochastic}.
Furthermore, they increasingly serve as an analytically convenient tool
for studying the unregularized optimal transport problem (see, e.g.,  \cite{leonard2012schrodinger,
gentil2017analogy, fathi2020proof, chewi2022entropic})
and they underlie some favorable statistical properties
that are currently under active investigation; see the works
\cite{mena2019statistical, genevay2019sample, del2020statistical,
  pooladian2021entropic, rigollet2022sample,
  pooladian2023minimax} and the
references therein.

Let us now define the entropic optimal transport cost.
Consider two probability measures, $\mu$ and $\nu$, both supported on
$\mathcal{X}$, and let $c: \mathcal{X}\times\mathcal{X} \to
[0,\infty)$ be a cost function.
The entropic Wasserstein distance with a regularization level
$\lambda>0$ is defined as

\begin{equation}
  \label{eq:entropic-Wasserstein-distance}
  T_{\lambda}(\mu,\nu)
  = \inf_{\gamma \in \Pi(\mu, \nu)}\mathbf{E}_{(X,Y) \sim \gamma}[c(X,Y)] +
  \lambda \kl{\gamma}{\mu \otimes \nu},
\end{equation}
where $\Pi(\mu,\nu)$ denotes the set of probability measures on $\mathcal{X}
\times \mathcal{X}$ with marginal distributions equal to $\mu$ and $\nu$,
and $\mathrm{KL}(\cdot,\cdot)$ is the Kullback-Leibler divergence.
When $\lambda \to 0$, the regularized cost $T_{\lambda}(\mu,\nu)$ converges
to the unregularized Wasserstein distance.
Various properties of entropic optimal transport can be found in the
recent lecture notes by \citet*{leonard2013survey}.

To develop efficiently computable approximations for Wasserstein barycenters, a
natural approach is to replace the unregularized Wasserstein cost with the
minimizer of the weighted sum of entropy-regularized costs.
This method was first explored
by \citet*{cuturi2014fast} and it has gained additional traction in the recent
years. There is some flexibility in the definition of \eqref{eq:entropic-Wasserstein-distance},
which arises from substituting the reference product measure $\mu \otimes \nu$ with
alternatives such as the Lebesgue measure. Consequently, various notions of
entropic barycenters have emerged in the literature, which can be
unified through the following optimization problem:
\begin{equation}
  \label{eq:doubly-entropic-barycenter-definition}
  \min_{\mu} \sum_{j=1}^{k} T_{\lambda}(\mu, \nu^{j}) + \tau\kl{\mu}{\piref}.
\end{equation}
Here $\nu^{1},\dots,\nu^{k}$ are the probability measures whose barycenter we
wish to compute and $w_{1}, \dots, w_{k}$ are positive weights that sum to one.
The inner regularization strength is denoted by $\lambda > 0$ while
$\tau > 0$ is the outer regularization strength.
The measure $\piref$ is an arbitrary reference measure, the support of which
dictates the support of the computed barycenter.
For instance, if we take $\piref$ to be a uniform measure on a particular
discretization of the underlying space, we are dealing with a fixed-support
setup. On the other hand, letting $\piref$ be the Lebesgue measure
puts us in the free-support setup.
We shall refer to the minimizer of \eqref{eq:doubly-entropic-barycenter-definition}
as the $(\lambda, \tau)$-barycenter, which exists and is unique due to the
strict convexity of the outer regularization penalty; however, uniqueness may no
longer holds when $\tau = 0$.

The objective \eqref{eq:doubly-entropic-barycenter-definition} was recently
studied in \cite{chizat2023doubly}; it also appeared earlier in
\cite{ballu2020stochastic} for the special case $\tau \geq \lambda$, where
  stochastic approximation algorithms were considered for the computation of
fixed-support barycenters.
In \cite[Section 1.3]{chizat2023doubly}, it is discussed how various choices
of $(\lambda, \tau)$ relate to Barycenters previously explored in the
literature. To provide a brief overview, $(0,0)$
are the unregularized Wasserstein barycenters studied in
\cite{agueh2011barycenters}. Inner-regularized barycenters $(\lambda, 0)$
introduce a shrinking bias; this can be seen already when $k=1$, in which
case the solution computes a maximum-likelihood deconvolution
\cite{rigollet2018entropic}.
The $(\lambda, \lambda)$-barycenters were considered in
\cite{cuturi2014fast, benamou2015iterative, cuturi2018semidual,
bigot2019penalization, kroshnin2019complexity}; they introduce a blurring
bias. Likewise, blurring bias is introduced by the outer-regularized
barycenters $(0,\tau)$, studied in \cite{bigot2019data, carlier2021entropic}.
The only case not covered via the
formulation \eqref{eq:doubly-entropic-barycenter-definition} appears to be
the one of debiased Sinkhorn barycenters \cite{ramdas2017wasserstein,
janati2020debiased}, for which an algorithm exists but without
computational guarantees.
Of particular interest are the $(\lambda,\lambda/2)$ barycenters: the choice
$\tau=\lambda/2$ for smooth densities yields approximation bias of order
$\lambda^{2}$, while the
choice $\tau=\lambda$ results in bias of order $\lambda$, which is
significantly larger than $\lambda^{2}$ in the regimes of interest.
This is a new notion of entropic barycenters that was unveiled in the analysis
of \cite{chizat2023doubly}.
We provide the first convergence guarantees for this type of barycenters.

The regularity, stability, approximation, and statistical sample complexity
properties of $(\lambda,\tau)$-barycenters were investigated in
\cite{chizat2023doubly}. However, the question of obtaining non-asymptotic
convergence guarantees for the computation of $(\lambda, \tau)$-barycenters with arbitrary
regularization parameters was not addressed therein.
In particular, the $(\lambda, \lambda/2)$ case, which has stood out due to its
compelling mathematical features, has not yet been addressed in the
existing literature. This gap is addressed by the present paper;
we summarize our contributions in the following section.

\subsection{Contributions}

The remainder of this paper is organized as follows:
Section~\ref{sec:background} provides the necessary background on
entropic optimal transport and a particular dual problem of the doubly
regularized entropic objective \eqref{eq:doubly-entropic-barycenter-definition}.
Section~\ref{sec:damped-sinkhorn} introduces a damped Sinkhorn iteration scheme
and complements it with convergence guarantees.
An approximate version of the algorithm together with convergence results and
implementation details is discussed in Section~\ref{sec:approximate-damped-sinkhorn}.
We summarize our key contributions:
\begin{enumerate}
  \item
    Lemma~\ref{lemma:suboptimality-to-kl}, presented in
    Section~\ref{sec:damped-sinkhorn}, demonstrates that bounds on the dual
    suboptimality gap for the dual problem \eqref{eq:doubly-entropic-dual},
    defined in Section~\ref{sec:doubly-entropic-barycenters}, can be
    translated into Kullback-Leibler divergence bounds between the
    $(\lambda,\tau)$-barycenter and the barycenters corresponding to dual-feasible variables.
    This translation enables us to formulate all our subsequent results in
    terms of optimizing the dual objective \eqref{eq:doubly-entropic-dual}.

  \item
    In Section~\ref{sec:damped-sinkhorn}, we introduce a damped Sinkhorn scheme
    (Algorithm~\ref{alg:exact}) that can be employed to optimize
    $(\lambda,\tau)$-barycenters for any choice of regularization parameters.
    The damping factor $\min(1,\tau/\lambda)$ accommodates the degrading
    smoothness properties of the dual objective \eqref{eq:doubly-entropic-dual}
    as a function of decreasing outer regularization parameter $\tau$.  The
    introduced damping of the Sinkhorn iterations is, in fact, necessary and it
    is one of our core contributions:
    undamped exact scheme can be experimentally shown to diverge as soon as
    $\tau < \lambda/2$.

  \item The main result of this paper is
    Theorem~\ref{thm:exact-scheme-convergence} proved in
    Section~\ref{sec:damped-sinkhorn}. It provides
    convergence guarantees for Algorithm~\ref{alg:exact} with
    arbitrary choice of regularization parameters $\lambda,\tau > 0$.
    This, in particular, results in the first algorithm with guarantees
    for computing $(\lambda,\lambda/2)$ barycenters. For smooth densities,
    these barycenters incur a
    bias of order $\lambda^{2}$ in contrast to the predominantly studied
    $(\lambda, \lambda)$ barycenters that incur bias of order $\lambda$.

  \item In Section~\ref{sec:approximate-damped-sinkhorn}, we describe
    Algorithm~\ref{alg:inexact}, an extension of Algorithm~\ref{alg:exact} that
    allows us to perform inaccurate updates.
    We formulate sufficient conditions on the inexact updates oracle under
    which the errors in the convergence analysis do not accumulate.
    Section~\ref{sec:inexact-oracle-implementation} details an implementation
    of this inexact oracle, based on approximate Monte Carlo sampling.

  \item Theorem~\ref{thm:inexact-scheme-convergence} proved in
    Section~\ref{sec:approximate-damped-sinkhorn} furnishes convergence
    guarantees for Algorithm~\ref{alg:inexact}.
    When combined with the implementation of the inexact oracle described in
    Section~\ref{sec:inexact-oracle-implementation}, this yields a provably
    convergent scheme for a grid-free computation of entropic Wasserstein
    barycenters between discrete distributions, provided sufficient regularity
    on the domain $\mathcal{X}$ and the cost function $c$.
\end{enumerate}

\section{Background and Notation}
\label{sec:background}

This section provides the background material on doubly regularized entropic
Wasserstein barycenters and introduces the notation used throughout the paper.
In the remainder of the paper, let $\mathcal{X}$ be a
compact and convex subset of $\mathbb{R}^{d}$ with a non-empty
interior. Let $\mathcal{P}(\mathcal{X})$ denote the set of probability measures on
$\mathcal{X}$ endowed with Borel sigma-algebra.
Let $c: \mathcal{X} \times \mathcal{X} \to [0, \infty)$ be a cost function such
that $c_{\infty}(\mathcal{X}) = \sup_{x,x' \in \mathcal{X}} c(x,x') < \infty$.
We denote by $\kl{\cdot}{\cdot}$ the Kullback-Leibler divergence,
$\|\cdot\|_{\mathrm{TV}}$ is the total-variation norm, and
$\|f\|_{\mathrm{osc}} = \sup_{x}f(x) - \inf_{x'}f(x')$ is the oscillation
norm.
Given two measures $\nu,\nu'$, the notation $\nu \ll \nu'$ denotes that $\nu$
is absolutely continuous with respect to the measure $\nu'$; in this case
$d\nu/d\nu'$ denotes the Radon-Nikodym derivative of $\nu$ with respect to
$\nu'$. Finally, throughout the paper $w$ denotes a vector of $k$ strictly
positive elements that sum to one.

\subsection{Entropic Optimal Transport}
\label{sec:entropic-ot}

For any $\mu,\nu \in \mathcal{P}(\mathcal{X})$ define the entropy regularized
optimal transport problem by
\begin{equation}
  \label{eq:entropic-ot-problem}
  T_{\lambda}(\mu,\nu)
  = \inf_{\gamma \in \Pi(\mu,\nu)}\mathbf{E}_{(X,Y) \sim \gamma}[c(X,Y)]
  + \lambda\kl{\gamma}{\mu \otimes \nu},
\end{equation}
where $\mathrm{KL}$ is the Kullback-Leibler divergence and
$\Pi(\mu,\nu) \subseteq \mathcal{P}(\mathcal{X}\otimes\mathcal{X})$ is the set
of probability measures such that for any $\gamma \in \Pi(\mu,\nu)$ and any
Borel subset $A$ of $\mathcal{X}$ it holds that
$\gamma(A \times \mathcal{X}) = \mu(A)$ and $\gamma(\mathcal{X} \times A) =
\nu(A)$.

Let $E_{\lambda}^{\mu,\nu} : L_{1}(\mu) \times L_{1}(\nu) \to \mathbb{R}$ be
the function defined by
\begin{align}
  \begin{split}
  \label{eq:E-dfn}
  E_{\lambda}^{\mu,\nu}(\phi, \psi)
  &= \mathbf{E}_{X \sim \mu}[\phi(X)] + \mathbf{E}_{Y \sim \nu}[\psi(Y)]
  \\&\quad\quad+ \lambda\left(1 -
    \int_{\mathcal{X}} \int_{\mathcal{X}}
    \exp\left(\frac{\phi(x) + \psi(y) - c(x,y)}{\lambda}\right)
    \nu(dy)\mu(dx)
  \right).
\end{split}
\end{align}
The entropic optimal transport problem \eqref{eq:entropic-ot-problem}
admits the following dual representation:
\begin{equation}
  \label{eq:entropic-ot-dual-representation}
  T_{\lambda}(\mu, \nu)
  = \max_{\phi,\psi}
  E_{\lambda}^{\mu,\nu}(\phi, \psi).
\end{equation}
For any $\psi$ define
\begin{equation}
  \phi_{\psi} \in \mathrm{argmax}_{\phi \in L_{1}(\mu)}
  E_{\lambda}^{\mu,\nu}(\phi, \psi).
\end{equation}
The solution is unique $\mu$-almost everywhere up to a constant; we fix a
particular choice
\begin{equation}
  \phi_{\psi}(x) =
  -\lambda\log\left(
    \int_{\mathcal{X}}
  \exp\left(\frac{\psi(y) - c(x,y)}{\lambda}\right)\nu(dy)\right).
\end{equation}
Likewise, we denote
$\psi_{\phi} = \mathrm{argmax}_{\psi \in L_{1}(\nu)}E_{\lambda}^{\mu,\nu}(\phi,
\psi)$ with the analogous expression to the one given above, interchanging the
roles of $\phi$ and $\psi$. Then,
the maximum in \eqref{eq:entropic-ot-dual-representation}
is attained by any pair $(\phi^{*}, \psi^{*})$ such that
$\phi^{*} = \phi_{\psi^{*}}$ and $\psi^{*} = \psi_{\phi^{*}}$; such a pair is
said to solve the Schr\"{o}dinger system and it is unique up to translations
$(\phi^{*}+a,\psi^{*}-a)$ by any constant $a\in \mathbb{R}$.
The optimal coupling that solves the primal problem
\eqref{eq:entropic-ot-problem} can be obtained from the pair
$(\phi^{*}, \psi^{*})$ via the primal-dual relation
\begin{equation}
  \label{eq:entropic-primal-dual-relation}
  \gamma^{*}(dx,dy) =
  \exp\left(\frac{\phi^{*}(x) + \psi^{*}(y) - c(x,y)}{\lambda}\right)
  \mu(dx)\nu(dy).
\end{equation}
We conclude this section by listing two properties of functions of the form
$\phi_{\psi}$. These properties will be used repeatedly throughout this paper.
First, for any $\psi$ we have
\begin{equation}
  \int_{\mathcal{X}}\int_{\mathcal{X}}
  \exp\left(\frac{\phi_{\psi}(x) + \psi(y) - c(x,y)}{\lambda}\right)
  \nu(dy)\mu(dx) = 1,
\end{equation}
which means, in particular, that for any $\psi$ we have
\begin{equation}
  \label{eq:semi-dual-entropic-ot}
  E_{\lambda}^{\mu,\nu}(\phi_{\psi}, \psi)
  = \mathbf{E}_{X \sim \mu}[\phi_{\psi}(X)] + \mathbf{E}_{Y \sim \nu}[\psi(Y)].
\end{equation}
The second property of interest is that for any $\psi$ and any $x,x' \in
\mathcal{X}$ it holds that
\begin{align}
  \phi_{\psi}(x) - \phi_{\psi}(x')
  &=
  -\lambda\log
  \frac{
    \int
    \exp\left(\frac{\psi(y) - c(x,y)}{\lambda}\right)\nu(dy)
  }
  {\int
    \exp\left(\frac{\psi(y) - c(x',y)}{\lambda}\right)\nu(dy)
  }
  \\
  &=
  -\lambda\log
  \frac{
    \int
    \exp\left(\frac{\psi(y) - c(x',y) + c(x',y) - c(x,y)}{\lambda}\right)\nu(dy)
  }
  {\int
    \exp\left(\frac{\psi(y) - c(x',y)}{\lambda}\right)\nu(dy)
  }
  \\
  &\leq \sup_{y \in \mathcal{X}} c(x', y) - c(x,y)
  \leq c_{\infty}(\mathcal{X}).
\end{align}
In particular, for any $\psi$ we have
\begin{equation}
  \label{eq:schroedinger-potentials-bounded}
  \|\phi_{\psi}\|_{\mathrm{osc}} = \sup_{x} \phi_{\psi}(x) -
  \inf_{x'}\phi_{\psi}(x')
  \leq c_{\infty}(\mathcal{X}).
\end{equation}

\subsection{Doubly Regularized Entropic Barycenters}
\label{sec:doubly-entropic-barycenters}

Let $\vecnu = (\nu^{1}, \dots, \nu^{k}) \in \mathcal{P}(\mathcal{X})^{k}$
be $k$ probability measures and let $w \in \mathbb{R}^{k}$ be a vector of
positive numbers that sum to one.
Given the inner regularization strength $\lambda > 0$ and the outer regularization
strength $\tau > 0$, the $(\lambda,\tau)$ barycenter $\mu_{\lambda,\tau} \in
  \mathcal{P}(\mathcal{X})$
of probability measures
$\vecnu$ with respect to the weights vector $w$ is defined as the unique
solution to the following optimization problem:
\begin{equation}
  \label{eq:doubly-regularized-barycenter-primal}
  \mu_{\lambda,\tau}
  = \mathrm{argmin}_{\mu \in \mathcal{P}(\mathcal{X})}\,
  \sum_{j=1}^{k}w_{j}
  T_{\lambda}(\mu,\nu^{j})
  + \tau\kl{\mu}{\piref},
\end{equation}
where $\piref \in \mathcal{P}(\mathcal{X})$ is a reference probability measure.

We will now describe how to obtain a concave dual maximization problem to the
primal problem \eqref{eq:doubly-regularized-barycenter-primal}, following along
the lines of \citet*[Section 2.3]{chizat2023doubly}, where the interested
reader will find a comprehensive justification of all the claims made in the
rest of this section.

First, using the semi-dual formulation of entropic optimal transport problem
\eqref{eq:semi-dual-entropic-ot}, we have, for each $j \in \{1,\dots,k\}$
\begin{equation}
  T_{\lambda}(\mu,\nu^{j})
  = \sup_{\psi^{j} \in L_{1}(\nu^{j})} \mathbf{E}_{X\sim\mu}[\phi_{\psi^{j}}(X)]
  + \mathbf{E}_{Y \sim \nu^{j}}[\psi^{j}(Y)].
\end{equation}
Denote $\vecpsi = (\psi^{1},\dots,\psi^{j}) \in L_{1}(\vecnu)$. Then, we may
rewrite the primal problem \eqref{eq:doubly-regularized-barycenter-primal} by
\begin{equation}
  \min_{\mu \in \mathcal{P}(X)} \max_{\vecpsi \in L_{1}(\vecnu)}
  \sum_{j=1}^{k}w_{j} \mathbf{E}_{Y \sim \nu^{j}}\big[\psi^{j}(Y)\big]
  +
  \mathbf{E}_{X\sim\mu}\big[\sum_{j=1}^{k}w_{j}\phi_{\psi^{j}}(X)\big]
  + \tau\kl{\mu}{\piref}.
\end{equation}
Interchanging $\min$ and $\max$, which is justified using compactness of
$\mathcal{X}$ as detailed in \cite{chizat2023doubly}, we obtain the dual optimization
objective $E_{\lambda,\tau}^{\vecnu, w} : L_{1}(\vecnu) \to \mathbb{R}$
defined by
\begin{align}
  \begin{split}
  \label{eq:doubly-entropic-dual}
    E_{\lambda,\tau}^{\vecnu, w}(\vecpsi)
    &= \min_{\mu \in \mathcal{P}(X)}
    \sum_{j=1}^{k}w_{j} \mathbf{E}_{Y \sim \nu^{j}}\big[\psi^{j}(Y)\big]
    +
    \mathbf{E}_{X\sim\mu}\big[\sum_{j=1}^{k}w_{j}\phi_{\psi^{j}}(X)\big]
    + \tau\kl{\mu}{\piref}.
    \\
    &=
    \sum_{j=1}^{k}w_{j} \mathbf{E}_{Y \sim \nu^{j}}\big[\psi^{j}(Y)\big]
    - \tau
    \log \int \exp\left(\frac{-\sum_{j=1}^{k}\phi_{\psi^{j}}(x)}
    {\tau}\right) \piref(dx).
  \end{split}
\end{align}
The infimum above is attained by the measure
\begin{equation}
  \label{eq:doubly-entropic-argmin-mu}
  \mu_{\vecpsi}(dx)
  = Z_{\vecpsi}^{-1}\exp\left(
    \frac{-\sum_{j=1}^{k}\phi_{\psi^{j}}(x)}{\tau}
  \right)\piref(dx),\quad
  Z_{\vecpsi} =
    \int \exp\left(\frac{-\sum_{j=1}^{k}\phi_{\psi^{j}}(x)}
    {\tau}\right) \piref(dx).
\end{equation}
To each dual variable $\vecpsi$ we associate
the marginal measures $\nu^{j}_{\vecpsi}(dy)$
defined for $j=1,\dots,k$ by
\begin{equation}
  \label{eq:psi-marginals}
  \nu^{j}_{\vecpsi}(dy)
  = \nu^{j}(dy)\int\exp\left(
    \frac{\phi_{\psi^{j}}(x) + \psi^{j}(y) - c(x,y)}{\lambda}
  \right)\mu_{\vecpsi}(dx).
\end{equation}
Finally, we mention that the objective
$E^{\vecnu, w}_{\lambda,\tau}$
is
concave and for any $\vecpsi, \vecpsi'$ it holds that
\begin{equation}
  \lim_{h \to 0}
  \frac{
    E^{\vecnu, w}_{\lambda,\tau}(\vecpsi + h\vecpsi')
    -
    E^{\vecnu, w}_{\lambda,\tau}(\vecpsi)
  }{h}
  =
  \sum_{j=1}^{k}w_{j}\left(
    \mathbf{E}_{\nu^{j}}[(\psi')^{j}] -
    \mathbf{E}_{\nu^{j}_{\vecpsi}}[(\psi')^{j}]
  \right)
  \frac{}{}.
\end{equation}
In particular, fixing any optimal dual variable $\vecpsi^{*}$, for any
$\vecpsi$ it holds using concavity of $E^{\vecnu,w}_{\lambda,\tau}$ that
\begin{equation}
  \label{eq:concave-suboptimality-gap}
  0 \leq E(\vecpsi^{*}) - E(\vecpsi)
  \leq \sum_{j=1}^{k}w_{k}\left(
    \mathbf{E}_{\nu^{j}}\left[(\psi^{*})^{j} - \psi^{j}\right]
    -
    \mathbf{E}_{\nu^{j}_{\vecpsi}}\left[
      (\psi^{*})^{j} - \psi^{j}
    \right]
  \right).
\end{equation}
This concludes our overview of the background material on
$(\lambda,\tau)$-barycenters.


\section{Damped Sinkhorn Scheme}
\label{sec:damped-sinkhorn}

This section introduces a damped Sinkhorn-based optimization scheme
(Algorithm~\ref{alg:exact}) and provides guarantees for its convergence
(Theorem~\ref{thm:exact-scheme-convergence}). Before describing the algorithm,
we make a quick detour to the following lemma, proved in
Appendix~\ref{sec:proof-of-suboptimality-to-kl-lemma},
which shows that the sub-optimality gap bounds on the dual objective
\eqref{eq:doubly-entropic-dual} can be transformed into corresponding bounds on
relative entropy between the $(\lambda,\tau)$-barycenter and the barycenter
associated to a given dual variable.

\begin{lemma}
  \label{lemma:suboptimality-to-kl}
  Fix any $\lambda,\tau > 0$ and $\vecnu,w$.
  Let $\vecpsi^{*}$ be the maximizer of dual problem
  $E^{\vecnu, w}_{\lambda,\tau}$ and let $\mu_{\boldsymbol{\psi}^{*}}$ be the
  corresponding minimizer of the primal objective
  \eqref{eq:doubly-regularized-barycenter-primal}.
  Then, for any $\vecpsi \in L_{1}(\vecnu)$ we have
  \begin{equation}
    \kl{\mu_{\boldsymbol{\psi}^{*}}}{\mu_{\vecpsi}}
    \leq \tau^{-1}(E^{\vecnu, w}_{\lambda, \tau}
    (\boldsymbol{\psi}^{*})
    -E^{\vecnu, w}_{\lambda, \tau}(\vecpsi)).
  \end{equation}
\end{lemma}

We now turn to describing an iterative scheme that ensures convergence of the
dual suboptimality gap to zero.
Let $\vecpsi_{t}$ be an iterate at time $t$. Then, we have
\begin{equation}
  E^{\vecnu,w}_{\lambda, \tau}(\vecpsi_{t})
  = L(\vecpsi_{t}, \vecphi_{t}, \mu_{t})
  = \sum_{j=1}^{k}w_{j}\mathbf{E}_{\nu^{j}}[\psi_{t}^{j}]
  -\mathbf{E}_{\mu_{t}}[\phi_{t}^{j}]
  +\tau\kl{\mu_{t}}{\piref},
\end{equation}
where
\begin{equation}
  \label{eq:maxi-minimization-steps}
  \phi^{j}
  = \mathrm{argmax}_{\phi}\,
  E^{\mu_{t-1},\nu^{j}}_{\lambda}(\phi,\psi^{j}_{t})
  \quad\text{and}\quad
  \mu_{t} = \mathrm{argmin}_{\mu}\, \bigg\{
    \mathbf{E}_{\mu}\big[\sum_{j}w_{j}\phi^{j}_{t}\big] +
  \tau\kl{\mu}{\piref}\bigg\}.
\end{equation}
In particular, when optimizing the dual objective $E^{\nu,w}_{\lambda, \tau}$,
every time the variable $\vecpsi_{t}$ is updated, it automatically triggers
the exact maximization/minimization steps defined in
\eqref{eq:maxi-minimization-steps}.
It is thus a natural strategy to fix $\vecphi_{t}$ and $\mu_{t}$ and perform
exact minimization on $\vecpsi$, which can be done in closed form:
\begin{equation}
  \label{eq:undamped-updates}
  \psi^{j}_{t+1} =
  \mathrm{argmax}_{\psi}\,
  E^{\mu_{t},\nu^{j}}_{\lambda}(\phi^{j}_{t},\psi)
  =
  \psi^{j}_{t} -\lambda\log\frac{d\nu_{t}^{j}}{d\nu^{j}},
\end{equation}
where $\nu_{j}^{t}$ denotes the marginal distribution $\nu^{j}_{\vecpsi_{t}}$
defined in \eqref{eq:psi-marginals}.
The update \eqref{eq:undamped-updates} performs a Sinkhorn update on each block
of variables $\psi^{j}$.
Together, the update \eqref{eq:undamped-updates} followed by
\eqref{eq:maxi-minimization-steps} results in the iterative Bregman projections
algorithm introduced in \cite{benamou2015iterative}. In
\cite{kroshnin2019complexity}, it was shown that this scheme converges for the
$(\lambda,\lambda)$-barycenters.
The analysis of \cite{kroshnin2019complexity} is built upon a
different dual formulation from the one considered in our work;
this alternative formulation is only available
when $\tau = \lambda$ \cite[Section 2.3]{chizat2023doubly}
and thus excludes the consideration of debiased barycenters $(\lambda, \lambda/2)$.

We have observed empirically that the iterates of the iterative Bregman
projections (i.e., the scheme of updates defined in \eqref{eq:undamped-updates}
and \eqref{eq:maxi-minimization-steps})
diverge whenever $\tau < \lambda/2$. Indeed, decreasing the outer
regularization parameter $\tau$ makes the minimization step in
\eqref{eq:maxi-minimization-steps} less stable. As a result, the cumulative
effect of performing the updates \eqref{eq:undamped-updates} and
\eqref{eq:maxi-minimization-steps} may result in a decrease in the value of
the optimization objective $E^{\vecnu,w}_{\lambda,\tau}$.

One of the main contributions of our work is to show that this bad behaviour
can be mitigated by damping the exact Sinkhorn updates
\eqref{eq:undamped-updates}. This leads to Algorithm~\ref{alg:exact} for which
convergence guarantees are provided in
Theorem~\ref{thm:exact-scheme-convergence} stated below.

\begin{algorithm}
  \caption{Exact Damped Sinkhorn Scheme}
  \label{alg:exact}
  \KwIn{regularization strengths $\lambda,\tau > 0$,
  reference measure $\piref$,number of iterations $T$ and
  $k$ marginal measures $\nu^{1},\dots,\nu^{k}$ with positive weights
  $w_{1}, \dots, w_{k}$ such that $\sum_{j=1}^{k}w_{j} = 1$.}
  \begin{enumerate}
    \item
      Set $\eta=\min(1, \tau/\lambda)$
      and initialize $(\psi^{j}_{0})=0$
      for $j \in \{1,\dots,k\}$.
  \item For $t=0,1\dots,T-1$ do
  \begin{enumerate}
    \item $\phi_{t}^{j}(x)
      \leftarrow
      - \lambda \log\int_{\mathcal{X}} \exp((\psi_{t}^{j}(y)-c(x,y))/\lambda)\nu^{j}(dy)$
      for $j \in \{1,\dots,k\}$
    \item $V_{t}(x) \leftarrow \sum_{j=1}^{k} w_j \phi^{j}(t)(x)$
    \item $Z_{t}  \leftarrow \int \exp(-V_{t}(x)/\tau)d \piref(dx)$
    \item $\mu_{t}(dx) \leftarrow Z_{t}^{-1}\exp(-V_{t}(x)/\tau)\piref(dx)$
    \item $
      \frac{d\nu^{j}_{t}}{d\nu^{j}}(y)
      \leftarrow
      \int\exp\left(
      \frac{\phi^{j}_{t}(x) + \psi^{j}_{t}(y) - c(x,y)}{\lambda}\right)\mu_{t}(dx)$
      for $j \in \{1,\dots,k\}$.
    \item $\psi^{j}_{t+1}(y) \leftarrow \psi^{j}_{t}(y) - \eta\lambda\log
      \frac{d\nu^{j}_{t}}{d\nu^{j}}(y)$
      for $j \in \{1,\dots,k\}$.
    \end{enumerate}
  \item Return $(\phi_{T}^{j}, \psi_{T}^{j})_{j=1}^{k}$.
  \end{enumerate}
\end{algorithm}

\begin{theorem}
  \label{thm:exact-scheme-convergence}
  Fix any $\lambda,\tau > 0$ and $\vecnu,w$.
  Let $\psi^{*}$ be the maximizer of dual problem
  $E^{\vecnu, w}_{\lambda,\tau}$.
  Let $(\vecpsi_{t})_{t \geq 0}$ be the sequence of iterates generated by
  Algorithm~\ref{alg:exact}.
  Then, for any $t \geq 1$ it holds that
  \begin{equation}
    E^{\vecnu, w}_{\lambda,\tau}(\vecpsi^{*})
    -
    E^{\vecnu, w}_{\lambda,\tau}(\vecpsi_{t})
    \leq \frac{2c_{\infty}(\mathcal{X})^{2}}{\min(\lambda,\tau)}
    \,\frac{1}{t}.
  \end{equation}
\end{theorem}

Our convergence analysis draws upon the existing analyses of Sinkhorn's
algorithm \cite{altschuler2017near, dvurechensky2018computational}, which in
turn are based on standard proof strategies in smooth convex
optimization (e.g., \cite[Theorem 2.1.14]{nesterov2018lectures}).
Concerning the proof of Theorem~\ref{thm:exact-scheme-convergence},
the main technical contribution of our work lies in the following
proposition proved in Appendix~\ref{sec:proof-of-one-step-improvement-proposition}.

\begin{proposition}
  \label{prop:one-step-improvement}
  Consider the setup of Theorem~\ref{thm:exact-scheme-convergence}.
  Then, for any integer $t \geq 0$ it holds that
  \begin{equation}
    E^{\vecnu,w}_{\lambda, \tau}(\vecpsi_{t+1})
    - E^{\vecnu,w}_{\lambda,\tau}(\vecpsi_{t})
    \geq
    \min\left(\tau, \lambda\right)
    \sum_{j=1}^{k}w_{j}\kl{\nu^{j}}{\nu^{j}_{t}}.
  \end{equation}
\end{proposition}

With Proposition~\ref{prop:one-step-improvement} at hand, we are ready to prove
Theorem~\ref{thm:exact-scheme-convergence}.

\begin{proof}[Proof of Theorem~\ref{thm:exact-scheme-convergence}]
  Denote $\delta_{t} = E^{\vecnu,w}_{\lambda,\tau}(\vecpsi^{*})
  - E^{\vecnu,w}_{\lambda,\tau}(\vecpsi_{t})$.
  We would like to relate the suboptimality gap $\delta_{t}$ to the increment
  $\delta_{t} - \delta_{t+1}$. To do this, we will first show that the iterates
  $\vecpsi_{t}$ have their oscillation norm bounded uniformly in $t$.
  Indeed, for any $j \in \{1,\dots,k\}$, any $t \geq 1$, and any $y \in
  \mathcal{X}$ we have
  \begin{align}
    \psi^{j}_{t}(y)
    = (1-\eta)\psi^{j}_{t-1}(y) + \eta\psi_{\phi^{j}_{t}}(y).
  \end{align}
  By \eqref{eq:schroedinger-potentials-bounded},
  $\psi_{\phi^{j}_{t}}$ has oscillation norm bounded by
  $c_{\infty}(\mathcal{X})$. Because $\psi^{j}_{0} = 0$ and $\eta \in (0,1]$,
  by induction on $t$ it follows that $\|\psi_{t}\|_{\mathrm{osc}} \leq
  c_{\infty}(\mathcal{X})$ for any $t \geq 0$.
  Combining the bound on the dual sub-optimality gap
  \eqref{eq:concave-suboptimality-gap} with Pinsker's inequality yields
  \begin{equation}
    \delta_{t} \leq 2c_{\infty}(\mathcal{X})\sum_{j=1}^{k}w_{j}\|\nu^{j} -
    \nu_{t}^{j}\|_{\mathrm{TV}}
    \leq \sqrt{2}c_{\infty}\sum_{j=1}^{k}w_{j}\sqrt{\kl{\nu^{j}}{\nu^{j}_{t}}}.
  \end{equation}
  Using concavity of the square root function,
  Proposition~\ref{prop:one-step-improvement}  yields for any $t \geq 0$
  \begin{equation}
    \delta_{t} - \delta_{t+1}
    \geq \min(\lambda,\tau) \sum_{j=1}^{k}
    w_{j}\kl{\nu^{j}}{\nu^{j}_{t}}
    \geq
    \frac{\min(\lambda,\tau)}{2c_{\infty}(\mathcal{X})^{2}} \delta_{t}^{2}.
  \end{equation}
  By Proposition~\ref{prop:one-step-improvement}, the sequence $\delta_{t}$ is
  non-increasing. Hence, dividing the above equality by
  $\delta_{t}\delta_{t+1}$ yields
  \begin{equation}
    \frac{1}{\delta_{t+1}} - \frac{1}{\delta_{t}} \geq
    \frac{\min(\lambda,\tau)}{2c_{\infty}(\mathcal{X})^{2}}.
  \end{equation}
  Telescoping the left hand side completes the proof.
\end{proof}

\section{Approximate Damped Sinkhorn Scheme}
\label{sec:approximate-damped-sinkhorn}

In this section, we extend the analysis of Algorithm~\ref{alg:exact} to an
approximate version of the algorithm. Then, in
Section~\ref{sec:inexact-oracle-implementation}, we describe how inexact
updates may be implemented via approximate random sampling, thus enabling
the computation of $(\lambda,\tau)$-barycenters in the free-support setting
with convergence guarantees.

Algorithm~\ref{alg:inexact} describes an inexact version of
Algorithm~\ref{alg:exact}. It replaces the damped Sinkhorn iterations of
Algorithm~\ref{alg:exact} via approximate updates computed by
an approximate Sinkhorn oracle -- a procedure that satisfies the properties listed
in Definition~\ref{dfn:approximate-sinkhorn-oracle}.

\begin{definition}[Approximate Sinkhorn Oracle]
  \label{dfn:approximate-sinkhorn-oracle}
  An $\varepsilon$-approximate Sinkhorn oracle is a procedure that given
  any $\vecpsi$ and any index $j \in
  \{1, \dots, k\}$, returns a Radon-Nikodym derivative
  $\frac{d\widehat{\nu}^{j}_{\vecpsi}}{d\nu^{j}}$
  of a measure
  $\widehat{\nu}^{j}_{\vecpsi} \ll \nu^{j}$ that satisfies the following
  properties:
  \begin{enumerate}
    \item $\frac{d\widetilde{\nu}^{j}_{\vecpsi}}{d\nu^{j}}$ is strictly positive
      on the support of $\nu^{j}$;
    \item $\|\widetilde{\nu}^{j}_{\vecpsi} - \nu^{j}_{\vecpsi}\|_{\mathrm{TV}} \leq
      \varepsilon/(2c_{\infty}(\mathcal{X}))$;
    \item
      $\mathbf{E}_{Y \sim
      \nu^{j}}[
        \frac{d\nu_{\vecpsi}^{j}}{d\widetilde{\nu}_{\vecpsi}^{j}}(Y)
      ]
      \leq 1 + \varepsilon^{2}/(2c_{\infty}(\mathcal{X})^{2})$;
    \item For any $\eta \in [0,1]$ and any $j \in \{1,\dots,k\}$ it holds that
      $\|\psi^{j}
      + \eta\lambda
      \log(d\widetilde{\nu}^{j}_{\vecpsi}/d\nu^{j})\|_{\mathrm{osc}}
      \leq (1-\eta)\|\psi^{j}\|_{\mathrm{osc}} +
      \eta c_{\infty}(\mathcal{X})$.
  \end{enumerate}
\end{definition}

\begin{algorithm}
  \caption{Approximate Damped Sinkhorn Scheme}
  \label{alg:inexact}
  \KwIn{error tolerance parameter $\varepsilon>0$, a function
  ``ApproximateSinkhornOracle'' satisfying properties listed in
  Definition~\ref{dfn:approximate-sinkhorn-oracle},
  regularization strengths $\lambda,\tau > 0$,
  reference measure $\piref$,number of iterations $T$ and
  $k$ marginal measures $\nu^{1},\dots,\nu^{k}$ with positive weights
  $w_{1}, \dots, w_{k}$ such that $\sum_{j=1}^{k}w_{j} = 1$.}
  \begin{enumerate}
    \item
      Set $\eta=\min(1, \tau/\lambda)$
      and initialize $(\psi^{j}_{0})=0$ for $j \in \{1,\dots,k\}$.
  \item For $t=0,1\dots,T-1$ do
  \begin{enumerate}
    \item
      $\frac{d\widehat{\nu}^{j}_{t}}{d\nu^{j}}(y)
      \leftarrow \mathrm{ApproximateSinkhornOracle}(
      \vecnu, \lambda, \tau, \vecpsi_{t}, \varepsilon, j)$
      for $j \in \{1,\dots,k\}$.
    \item $\psi^{j}_{t+1}(y) \leftarrow \psi^{j}_{t}(y) - \eta\lambda\log
      \frac{d\widehat{\nu}^{j}_{t}}{d\nu^{j}}(y)$
      for $j \in \{1,\dots,k\}$.
    \end{enumerate}
  \item Return $(\phi_{T}^{j}, \psi_{T}^{j})_{j=1}^{k}$.
  \end{enumerate}
\end{algorithm}

The following theorem shows that Algorithm~\ref{alg:inexact}
enjoys the same convergence guarantees as Algorithm~\ref{alg:exact}
up to the error tolerance of the procedure used to implement the approximate
updates. A noteworthy aspect of the below theorem is that the error does not
accumulate over the iterations.

\begin{theorem}
  \label{thm:inexact-scheme-convergence}
  Fix any $\lambda,\tau > 0$ and $\vecnu,w$.
  Let $\psi^{*}$ be the maximizer of dual problem
  $E^{\vecnu, w}_{\lambda,\tau}$.
  Let $(\widetilde{\vecpsi}_{t})_{t \geq 0}$ be the sequence of iterates generated by
  Algorithm~\ref{alg:inexact} with the accuracy parameter $\varepsilon \geq 0$.
  Let $T = \min\{t : E^{\vecnu, w}_{\lambda,\tau}(\vecpsi^{*}) -
  E^{\vecnu, w}_{\lambda,\tau}(\widetilde{\vecpsi}_{t}) \leq 2\varepsilon\}$.
  Then, for any $t \leq T$ it holds that
  \begin{equation}
    E^{\vecnu, w}_{\lambda,\tau}(\vecpsi^{*})
    -
    E^{\vecnu, w}_{\lambda,\tau}(\widetilde{\vecpsi}_{t})
    \leq 2\varepsilon +
    \frac{2c_{\infty}(\mathcal{X})^{2}}{\min(\lambda,\tau)}
    \,\frac{1}{t}.
  \end{equation}
\end{theorem}

The proof of the above theorem can be found in
Appendix~\ref{sec:inexact-algorithm-convergence-proof}.

\subsection{Implementing the Approximate Sinkhorn Oracle}
\label{sec:inexact-oracle-implementation}
In this section, we show that the approximate Sinkhorn oracle (see
Definition~\ref{dfn:approximate-sinkhorn-oracle}) can be implemented using
approximate random sampling when the marginal distributions $\nu^{j}$ are
discrete. To this end, fix the regularization parameters $\lambda, \tau > 0$,
the weight vector $w$, and consider a set of $k$ discrete marginal distributions
$$
  \nu^{j} = \sum_{l=1}^{m_{j}}\nu^{j}(y^{j}_{l})\delta_{y^{j}_{l}},
$$
where $\delta_{x}$ is the Dirac measure located at $x$ and $\nu^{j}(y^{j}_{l})$
is equal to the probability of sampling the point $y^{j}_{l}$ from measure
$\nu^{j}$.
We denote the total cardinality of the support of all measures $\nu^{j}$ by
$$
  m = \sum_{j=1}^{m}m_{j}.
$$
Fix any $\vecpsi \in L_{1}(\vecnu)$. Suppose we are given access to
$n$ i.i.d.\ samples $X_{1},\dots,X_{n}$ from a probability measure
$\mu'_{\vecpsi}$ that satisfies
\begin{equation}
  \|\mu_{\vecpsi} - \mu_{\vecpsi}'\|_{\mathrm{TV}} \leq \varepsilon_{\mu}.
\end{equation}
Then, for $j=1,\dots,k$ and $l=1,\dots,m_{j}$ consider
\begin{equation}
  \widehat{\nu}^{j}(y^{j}_{i})
  =
  \nu^{j}(y^{j}_{i})\frac{1}{n}
  \sum_{i=1}^{n}\exp\left(
    \frac{\phi_{\psi^{j}}(X_{i}) + \psi^{j}(y) - c(x,y)}{\lambda}\right)
\end{equation}
and for any parameter $\zeta \in (0,1/2]$ define
\begin{equation}
  \label{eq:approximate-oracle}
  \widetilde{\nu}^{j} = (1-\zeta)\widehat{\nu}^{j} + \zeta\nu^{j}.
\end{equation}
We claim that $\widetilde{\nu}^{j}$ implements the approximate Sinkhorn oracle
with accuracy parameter arbitrarily close to $\sqrt{\varepsilon_{\mu}}$
provided that $n$ is large enough.
This is shown in the following lemma, the proof of which can be found in
Appendix~\ref{sec:proof-of-approximate-oracle-implementation}.

\begin{lemma}
  \label{lemma:approximate-oracle-implementation}
  Fix any $\delta \in (0,1)$ and consider the setup described above.
  With probability at least $1-\delta$,
  for each $j \in \{1,\dots,k\}$ it holds simultaneously that
  the measure
  $\widetilde{\nu}^{j}$ defined in \eqref{eq:approximate-oracle}
  satisfies all the properties listed in
  Definition~\ref{dfn:approximate-sinkhorn-oracle}
  with accuracy parameter
  $$
    \varepsilon_{j} \leq
    c_{\infty}(\mathcal{X})
    \Bigg(
      2\zeta
      + \frac{1}{\zeta}
      m_{j}\varepsilon_{\mu}
      +
      \frac{1}{\zeta}
      m_{j}
        \sqrt{\frac{2\log\left(\frac{2m}{\delta}\right)}{n}}
    \Bigg)^{1/2}.
  $$
\end{lemma}

The above lemma shows that a step of Algorithm~\ref{alg:inexact} can be implemented
  provided access to i.i.d.\ sampling from some measure $\mu_{\vecpsi}'$ close
  to $\mu_{\vecpsi}$ in total variation norm, where $\vecpsi$ is an arbitrary
  iterate of Algorithm~\ref{alg:inexact}. The remainder of this section is
  dedicated to showing that this can be achieved by sampling via Langevin Monte
Carlo.

Henceforth, fix $\piref$ to be the Lebesgue measure on $\mathcal{X}$, which
corresponds to the free-support barycenters setup.
Then, for any $\vecpsi$ we have
$$
  \mu_{\vecpsi}(dx)
  \propto
  \mathbb{1}_{\mathcal{X}}
  \exp(-V_{\vecpsi}(x)/\tau)dx,
  \quad\text{where}\quad V_{\vecpsi}(x) =
  \sum_{j=1}^{k}w_{j}\phi^{j}_{\psi^{j}},
$$
where $\mathbb{1}_{\mathcal{X}}$ is equal to one on $\mathcal{X}$
and zero everywhere else.
It follows by \eqref{eq:schroedinger-potentials-bounded}
that $\|V_{\vecpsi}\|_{\mathrm{osc}} \leq c_{\infty}(\mathcal{X})/\tau$.
Further, let $\mathrm{diam}{\mathcal{X}} = \sup_{x,x' \in \mathcal{X}} \|x - x'\|_{2}$.
By the convexity of $\mathcal{X}$, the uniform measure on $\mathcal{X}$
satisfies the logarithmic Sobolev inequality (LSI) with constant
$\mathrm{diam}(\mathcal{X})^{2}/4$ (cf.\ \cite{lehec2021langevin}).
Hence, by the Holley-Stroock perturbation argument
\cite{holley1986logarithmic}, the measure $\mu_{\vecpsi}$ satisfies LSI with
constant at most
$\exp\left(2c_{\infty}(\mathcal{X})/\tau\right)\mathrm{diam}(\mathcal{X})^{2}/4
< \infty$.

It is well-established that Langevin Monte Carlo algorithms offer convergence
guarantees for approximate sampling from a target measure subject to functional
inequality constraints provided additional conditions hold such as
the smoothness of the function $V_{\vecpsi}$.
However, such guarantees do not directly apply to the measure $\mu_{\vecpsi}$
due to its constrained support.
Instead, it is possible to approximate $\mu_{\vecpsi}$
arbitrarily well in total variation norm by a family of measures
$(\mu_{\vecpsi,\sigma})_{\sigma > 0}$ (see Appendix~\ref{sec:langevin-sampling-guarantees} for details)
supported on all of $\mathbb{R}^{d}$.
Tuning the parameter $\sigma$ allows us to
trade-off between the approximation quality of $\mu_{\vecpsi,\sigma}$ and its
LSI constant. Crucially, standard sampling guarantees for Langevin Monte Carlo
(e.g., \cite{vempala2019rapid}) apply to the regularized measures
$\mu_{\vecpsi,\sigma}$, which leads to provable guarantees for an
implementation of Algorithm~\ref{alg:inexact}, thus furnishing
the first convergence guarantees for computation of Wasserstein barycenters
in the free support setup; see
Theorem~\ref{thm:inexact-algorithm-implementation} stated below.

The above approximation argument applies to any cost function $c$ that is Lipschitz on
$\mathcal{X}$ and exhibits quadratic growth at infinity. For the sake of
simplicity, we consider the quadratic cost $c(x,y) = \|x-y\|_{2}^{2}$.
The exact problem setup where we are able to obtain computational guarantees
for free-support barycenter computation via Langevin Sampling is formalized
below.

\begin{setup}
  \label{setup:ball-and-squared-loss}
  Consider the setting described at the beginning of
  Section~\ref{sec:inexact-oracle-implementation}. In addition, suppose that
  \begin{enumerate}
    \item the reference measure $\piref(dx) = \mathbb{1}_{\mathcal{X}}dx$
      is the Lebesgue measure supported on $\mathcal{X}$
      (free-support setup);
    \item it holds that
      $\mathcal{X} \subseteq \mathcal{B}_{R} = \{ x \in \mathbb{R}^{d} : \|x\|_{2} \leq R\}$ for some constant $R < \infty$;
    \item the cost function
      $c: \mathbb{R}^{d} \times \mathbb{R}^{d} \to [0, \infty)$ is defined
      by $c(x,y) = \|x - y\|_{2}^{2}$;
    \item for any $\vecpsi$ we have
      access to a stationary point $x_{\vecpsi}$ of $V_{\vecpsi}$ over $\mathcal{X}$.
  \end{enumerate}
\end{setup}
The last condition can be implemented in polynomial time
using a first order gradient method. For our purposes, this condition is needed
to obtain a good initialization point for the Unadjusted Langevin Algorithm
following the explanation in \cite[Lemma 1]{vempala2019rapid}; see
Appendix~\ref{sec:langevin-sampling-guarantees} for further details.

We now proceed to the main result of this section, the proof of which can be
found in Appendix~\ref{sec:langevin-sampling-guarantees}.
The following theorem provides the first provably convergent method for
computing Wasserstein barycenters in the free-support setting.
We remark that a stochastic approximation argument of a rather different flavor
used to
compute fixed-support Wasserstein barycenters (for $\tau \geq \lambda$) has
been previously analyzed in \cite{ballu2020stochastic}.

\begin{theorem}
  \label{thm:inexact-algorithm-implementation}
  Consider the setup described in Problem Setting~\ref{setup:ball-and-squared-loss}.
  Then, for any confidence parameter $\delta \in (0,1)$ and any
  accuracy parameter $\varepsilon > 0$, we can simulate a step
  of Algorithm~\ref{alg:inexact} with success probability at least $1-\delta$
  in time polynomial in
  \begin{equation}
    \varepsilon^{-1}, d, R, \exp(R^{2}/\tau),
    (Rd^{-1/4})^{d}, \tau^{-1}, \lambda^{-1}, d, m, \log(m/\delta).
  \end{equation}
  In particular, an $\varepsilon$-approximation of the
  $(\lambda, \tau)$-Barycenter can be obtained within the same computational
  complexity.
\end{theorem}

Comparing the above guarantee with the discussion following the statement of
Lemma~\ref{lemma:approximate-oracle-implementation},
we see an additional polynomial dependence on $(Rd^{-1/4})^{d}$ (note that for
$R \leq d^{1/4}$ this term disappears). We believe this term to be an
artefact of our analysis appearing due to the approximation argument described
above. Considering the setup with $R \leq d^{1/4}$, the
running time of our algorithm depends exponentially in $R^{2}/\tau$.

We conclude with two observations. First, since approximating
Wasserstein barycenters is generally NP-hard \cite{altschuler2022wasserstein},
an algorithm with polynomial dependence on all problem parameters does not
exist if $\mathrm{P}\neq\mathrm{NP}$. Second, notice that
computing an $\varepsilon$ approximation of $(\lambda,\tau)$-Barycenter can be
done in time polynomial in $\varepsilon^{-1}$.
This should be contrasted with numerical schemes based on discretizations of the
set $\mathcal{X}$, which would, in general, result in computational complexity
of order $(R/\varepsilon)^{d}$ to reach the same accuracy.

\section{Conclusion}
\label{sec:conclusion}

We introduced algorithms to compute doubly regularized entropic
Wasserstein barycenters and studied their computational complexity, both in the
fixed-support and in the free-support settings. Although a naive adaptation of
the usual alternate maximization scheme from~\cite{benamou2015iterative} to our
setting leads to diverging iterates (at least for small values of $\tau$), our
analysis shows that it is sufficient to damp these iterations to get a
converging algorithm.

While we have focused on the problem of barycenters of measures, we note that
the idea of entropic regularization is pervasive in other applications of
optimal transport. There, the flexibility offered by the double entropic
regularization may prove to be useful as well, and we believe that our damped
algorithm could be adapted to these more general settings.

\bibliographystyle{abbrvnat}
\bibliography{references}

\appendix
\clearpage

\section{Proof of Lemma~\ref{lemma:suboptimality-to-kl}}
\label{sec:proof-of-suboptimality-to-kl-lemma}
To simplify the notation throughout this proof, for each $j \in \{1,\dots,k\}$
denote $\phi^{j} = \phi_{\psi^{j}}$. We have
\begin{align}
  \label{eq:suboptimality-to-kl-proof-first-equation}
  E^{\vecnu,w}_{\lambda, \tau}(\vecpsi^{*})
  - E^{\vecnu,w}_{\lambda,\tau}(\vecpsi)
  =
  \sum_{j=1}^{k}w_{j}
  \mathbf{E}_{Y \sim \nu^{j}}\left[(\psi^{*})^{j}(Y) -
  \psi^{j}(Y)\right]
  - \tau
  \log
  \frac{Z_{\vecpsi^{*}}}{Z_{\vecpsi}}.
\end{align}
Observe that for any $x \in \mathcal{X}$ it holds that
\begin{align}
  \frac{d\mu_{\vecpsi}}{d\mu_{\vecpsi^{*}}}(x)
  =
  \frac{Z_{\vecpsi^{*}}}{Z_{\vecpsi}}
  \exp\left(
    -\frac{
      \sum_{j=1}^{k}w_{j}(\phi^{j}(x) - (\phi^{*})^{j}(x))
    }{\tau}
  \right).
\end{align}
Hence,
\begin{align}
  \tau\log
  \frac{Z_{\vecpsi^{*}}}{Z_{\vecpsi}}
  &=
  \tau\log\mathbf{E}_{X \sim \mu_{\vecpsi^{*}}}\left[
     \frac{Z_{\vecpsi^{*}}}{Z_{\vecpsi}}
  \right]
  \\
  &=
  \tau\log\mathbf{E}_{X \sim \mu_{\vecpsi^{*}}}\left[
    \frac{d\mu_{\vecpsi}}{d\mu_{\vecpsi^{*}}}(x)
    \exp\left(
      \frac{
        \sum_{j=1}^{k}w_{j}(\phi^{j}(x) - (\phi^{*})^{j}(x))
      }{\tau}
    \right)
  \right]
  \\
  &=
  \tau\log\mathbf{E}_{X \sim \mu_{\vecpsi}}\left[
    \exp\left(
      \frac{
        \sum_{j=1}^{k}w_{j}(\phi^{j}(x) - (\phi^{*})^{j}(x))
      }{\tau}
    \right)
  \right]
  \\
  &=
  \sup_{\mu \ll \mu_{\vecpsi}}\left\{
    \mathbf{E}_{X \sim \mu}\left[
        \sum_{j=1}^{k}w_{j}(\phi^{j}(x) - (\phi^{*})^{j}(x))
    \right]
    - \tau \kl{\mu}{\mu_{\vecpsi}}
  \right\},
  \label{eq:kl-log-mgf-duality}
\end{align}
where in the final expression we have applied the
Donsker-Varadhan variational principle (i.e., convex-conjugate
duality between KL-divergence and cumulant generating functions); therein, the
supremum runs over probability measures $\mu$ absolutely continuous with
respect to $\mu_{\vecpsi}$, and it is attained by $\mu$ defined as
\begin{align}
  \mu(dx)
  &\propto
  \exp\left(
    \frac{1}{\tau}\sum_{j=1}^{k}w_{j}(\phi^{j}(x) - (\phi^{*})^{j}(x))
  \right)\mu_{\vecpsi}(dx)
  \\
  &\propto
  \exp\left(
    \frac{1}{\tau}\sum_{j=1}^{k}w_{j}(\phi^{j}(x) - (\phi^{*})^{j}(x))
  \right)
  \exp\left(
    -\frac{1}{\tau}\sum_{j=1}^{k}w_{j}\phi^{j}(x)
  \right)\piref(dx)
  \\
  &\propto
  \exp\left(
    -\frac{1}{\tau}\sum_{j=1}^{k}w_{j}(\phi^{*})^{j}(x)
  \right)
  \piref(dx)
  = \pi_{\vecpsi^{*}}(dx).
\end{align}
That is, the supremum in \eqref{eq:kl-log-mgf-duality} is attained by $\mu =
\mu_{\vecpsi^{*}}$. Hence, the identity
\eqref{eq:suboptimality-to-kl-proof-first-equation} becomes
\begin{align}
  &E^{\vecnu,w}_{\lambda, \tau}(\vecpsi^{*})
  - E^{\vecnu,w}_{\lambda,\tau}(\vecpsi)
  \\
  &=
  \sum_{j=1}^{k}
  w_{j}
  \mathbf{E}_{Y \sim \nu^{j}}\left[(\psi^{*})^{j}(Y) - \psi^{j}(Y)\right]
  -
  \mathbf{E}_{X \sim \mu_{\vecpsi^{*}}}\bigg[
      \sum_{j=1}^{k}w_{j}(\phi^{j}(X) - (\phi^{*})^{j}(X))
  \bigg]
  \\
  &\quad\quad
  + \tau \kl{\mu_{\vecpsi^{*}}}{\mu_{\vecpsi}}
  \\
  &=
  \sum_{j=1}^{k}
  w_{j}
  \left(
    \mathbf{E}_{Y \sim \nu^{j}}\left[(\psi^{*})^{j}(Y) -
    \psi^{j}(Y)\right]
    +
    \mathbf{E}_{X \sim \mu_{\vecpsi^{*}}}\left[
      (\phi^{*})^{j}(X))
      - \phi^{j}(X)
    \right]
  \right)
  + \tau \kl{\mu_{\vecpsi^{*}}}{\mu_{\vecpsi}}
  \\
  &\geq
  \tau \kl{\mu_{\vecpsi^{*}}}{\mu_{\vecpsi}},
\end{align}
where the final inequality follows by noting that for each $j$ the optimality
of the pair $((\phi^{*})^{j}, (\psi^{*})^{j})$ for the entropic optimal transport
dual objective $E^{\mu_{\vecpsi^{*}}, \nu^{j}}_{\lambda}$ implies
that
\begin{align}
    &\mathbf{E}_{Y \sim \nu^{j}}\left[(\psi^{*})^{j}(Y) -
    \psi^{j}(Y)\right]
    +
    \mathbf{E}_{X \sim \mu_{\vecpsi^{*}}}\left[
      (\phi^{*})^{j}(X)
      - \phi^{j}(X)
    \right]
    \\
    &=
    E^{\mu, \nu^{j}}_{\lambda}((\phi^{*})^{j}, (\psi^{*})^{j})
    - E^{\mu,\nu^{j}}_{\lambda}(\phi^{j}, \psi^{j})
    \geq 0.
\end{align}
The proof of Lemma~\ref{lemma:suboptimality-to-kl} is complete.\hfill\qed

\section{Proof of Proposition~\ref{prop:one-step-improvement}}
\label{sec:proof-of-one-step-improvement-proposition}
Recall that for any non-negative integer $t$ we have
\begin{equation}
  \mu_{t}(dx) = Z_{t}^{-1}
  \exp\left(-\frac{\sum_{j=1}^{k}
  w_{j}\phi^{j}_{t}(x)}{\tau}\right) \piref(dx).
\end{equation}
where $Z_{t}$ is the normalizing constant defined by
\begin{equation}
  Z_{t} = \int_{\mathcal{X}}
  \exp\left(-\frac{\sum_{j=1}^{k}
  w_{j}\phi^{j}_{t}(x)}{\tau}\right) \piref(dx).
\end{equation}
With the notation introduced above, we have
\begin{equation}
  E(\vecpsi_{t})
  = \sum_{j=1}^{k}w_{j}\mathbf{E}_{Y \sim \nu^{j}}\left[\psi^{j}_{t}(Y)\right]
  - \tau \log Z_{t}.
\end{equation}
Hence,
\begin{align}
  E(\vecpsi_{t+1}) - E(\vecpsi_{t})
  &=
  \sum_{j=1}^{k}
  w_{j}\mathbf{E}_{Y \sim \nu^{j}}\left[\psi^{j}_{t+1}(Y) - \psi^{j}_{t}(Y)\right]
  - \tau \log \frac{Z_{t+1}}{Z_{t}}.
  \\
  &=
  \eta\lambda\sum_{j=1}^{k}
  w_{j}\mathbf{E}_{Y \sim \nu^{j}}\left[\log\frac{d\nu^{j}}{d\nu^{j}_{t}}(Y)\right]
  - \tau \log \frac{Z_{t+1}}{Z_{t}}.
  \\
  &=
  \min\left(\lambda, \tau\right)
  \sum_{j=1}^{k}
  w_{j}
  \kl{\nu^{j}}{\nu^{j}_{t}}
  - \tau \log \frac{Z_{t+1}}{Z_{t}}.
\end{align}
Therefore, to prove Proposition~\ref{prop:one-step-improvement} it suffices to
show that the inequality
\begin{equation}
  \label{eq:prop-one-step-improvement-sufficient}
  \log \frac{Z_{t+1}}{Z_{t}} \leq 0
\end{equation}
holds for any $t \geq 0$.
We will complete the proof of Proposition~\ref{prop:one-step-improvement}
using the following lemma, the proof of which is
deferred to the end of this section.

\begin{lemma}
  \label{lemma:log-Z-ratio-bound}
  Let $(\vecpsi_{t})_{t \geq 0}$ be any sequence of the form
  $$
    \psi_{t+1}^{j} = \psi_{t}^{j} + \eta\lambda\log(\Delta^{j}_{t}),
  $$
  where for $j \in \{1,\dots,k\}$,  $(\Delta_{t}^{j})_{t \geq 0}$ is an
  arbitrary sequence of strictly positive functions
  and $\eta = \min(1, \tau/\lambda)$.
  Then, for any $t \geq 0$ it holds that
  \begin{equation}
    \tau\log\frac{Z_{\vecpsi_{t+1}}}{Z_{\vecpsi_{t}}}
    \leq
    \min(\lambda, \tau)
    \log
    \sum_{j=1}^{k}w_{j}
    \mathbf{E}_{Y \sim \nu^{j}_{\vecpsi_{t}}}\left[
      \Delta^{j}_{t}(Y)
    \right].
  \end{equation}
\end{lemma}

To complete the proof of Proposition~\ref{prop:one-step-improvement},
we will apply the above lemma with $\Delta_{t}^{j} = \log
\frac{d\nu^{j}}{d\nu_{t}^{j}}$. Indeed, we have
\begin{align}
  \tau \log \frac{Z_{t+1}}{Z_{t}}
  &\leq
  \min(\lambda,\tau)
  \log
  \sum_{j=1}^{k}w_{j}
  \mathbf{E}_{Y \sim \nu^{j}_{t}}\left[
    \frac{d\nu^{j}}{d\nu_{t}^{j}}(Y)
  \right]
  \\
  &=
  \min(\lambda,\tau)
  \log
  \sum_{j=1}^{k}w_{j}
  \mathbf{E}_{Y \sim \nu^{j}}\left[
    1
  \right]
  \\
  &=
  0.
\end{align}
By \eqref{eq:prop-one-step-improvement-sufficient}, the proof of
Proposition~\ref{prop:one-step-improvement} is complete.
\hfill\qed

\subsection{Proof of Lemma~\ref{lemma:log-Z-ratio-bound}}

We will break down the proof
with the help of the following lemma, the proof of which can be found in
Section~\ref{sec:proof-of-Zt-ratio-lemma}.

\begin{lemma}
  \label{lemma:zt-ratio}
  For any sequence $(\vecpsi_{t})_{t \geq 0}$ and any $t \geq 0$
  it holds that
  \begin{equation}
    \log
    \frac{Z_{\vecpsi_{t+1}}}{Z_{\vecpsi{t}}}
    \leq
    \begin{cases}
      \frac{\lambda}{\tau}
      \log
      \sum_{j=1}^{k}
      w_{j}
      \mathbf{E}_{X \sim \mu_{t}}\left[
          \exp\left(\frac{-\phi^{j}_{t+1}(X)+\phi^{j}_{t}(X)}{\tau}
          \right)^{\tau/\lambda}
        \right]
        & \text{ if }\tau \geq \lambda,\\
      \log
      \sum_{j=1}^{k}
      w_{j}
      \mathbf{E}_{X \sim \mu_{t}}\left[
          \exp\left(\frac{-\phi^{j}_{t+1}(X)+\phi^{j}_{t}(X)}{\tau}
          \right)
        \right]
        & \text{ if }\tau < \lambda,
    \end{cases}
  \end{equation}
  where $\vecphi_{t} = \vecphi_{\vecpsi_{t}}$
  and $\mu_{t}(dx) =
  Z_{\vecpsi_{t}}^{-1}
  \exp(-\sum_{j=1}^{k}w_{j}\phi_{t}^{j}(x)/\tau)\piref(dx)$.
\end{lemma}

Observe that the sequence $(\vecpsi_{t})_{t\geq0}$ of the form stated in
Lemma~\ref{lemma:log-Z-ratio-bound} satisfies, for any
for any $j \in \{1,\dots,k\}$ and any $t \geq 0$,
\begin{equation}
    \exp\left(
      \frac{-\phi^{j}_{t+1}+\phi^{j}_{t}}{\tau}
    \right)
    =
    \exp\left(
      -\frac{\lambda}{\tau}
      \log\frac{d\mu_{t}}{d\tilde{\mu}^{j}_{t}}
    \right)
    =
    \left(\frac{d\tilde{\mu}^{j}_{t}}{d\mu_{t}}\right)^{\lambda/\tau},
\end{equation}
where
\begin{align}
  \frac{d\tilde{\mu}^{j}_{t}}{d\mu_{t}}(x)
  &=
  \int_{\mathcal{X}}
  \nu(dy)\exp\left(\frac{\psi^{j}_{t+1}(y) + \phi^{j}_{t}(x) - c(x,y)}{\lambda}\right)
  \\
  &=
  \int_{\mathcal{X}}
  \nu^{j}(dy) \Delta_{t}^{j}(y)^{\eta}\exp\left(
    \frac{\psi^{j}_{t}(y) + \phi^{j}_{t}(x) - c(x,y)}{\lambda}\right).
\end{align}
Hence, by Lemma~\ref{lemma:zt-ratio}
we have
\begin{align}
  &\log\frac{Z_{\vecpsi_{t+1}}}{Z_{\vecpsi_{t}}}
  \\
  &\leq
  \frac{1}{\tau}\min(\lambda, \tau)
  \sum_{j=1}^{k}
  w_{j}
  \mathbf{E}_{X \sim \mu_{t}}\Bigg[
  \\&\quad\quad
    \left(
    \int_{\mathcal{X}}
    \nu^{j}(dy)\Delta_{t}^{j}(y)^{\eta}\exp\left(
      \frac{\psi^{j}_{t}(y) + \phi^{j}_{t}(X) - c(X,y)}{\lambda}\right)
    \right)^{\max(1, \lambda/\tau)}
  \Bigg].
  \label{eq:Zt-ratio-bound-before-case-splitting}
\end{align}
We split the remaining proof into two cases:
$\tau \geq \lambda$ and $\tau < \lambda$.

\paragraph{The case \texorpdfstring{$\tau \geq \lambda$}{tau >= lambda}.}
When $\tau \geq \lambda$, we have $\max(1,\lambda/\tau) = 1$ and $\eta =
\min(1,\tau/\lambda) = 1$. Thus,  \eqref{eq:Zt-ratio-bound-before-case-splitting}
yields
\begin{align}
  &\log\frac{Z_{\vecpsi_{t+1}}}{Z_{\vecpsi_{t}}}
  \\
  &\leq
  \frac{1}{\tau}\min(\lambda, \tau)
  \log
  \sum_{j=1}^{k}
  w_{j}
  \mathbf{E}_{X \sim \mu_{t}}\left[
    \int_{\mathcal{X}}
    \nu^{j}(dy)\Delta_{t}^{j}(y)\exp\left(
      \frac{\psi^{j}_{t}(y) + \phi^{j}_{t}(X) - c(X,y)}{\lambda}\right)
  \right]
  \\
  &=
  \frac{1}{\tau}\min(\lambda, \tau)
  \log
  \sum_{j=1}^{k}
  w_{j}
  \left[
    \int_{\mathcal{X}}
    \mu_{t}(dx)
    \int_{\mathcal{X}}
    \nu^{j}(dy)\Delta_{t}^{j}(y)\exp\left(
      \frac{\psi^{j}_{t}(y) + \phi^{j}_{t}(X) - c(X,y)}{\lambda}\right)
  \right]
  \\
  &=
  \frac{1}{\tau}\min(\lambda, \tau)
  \log
  \sum_{j=1}^{k}
  w_{j}
  \left[
    \int_{\mathcal{X}}
    \Delta_{t}^{j}(y)\nu^{j}(dy)
    \int_{\mathcal{X}}
    \exp\left(
      \frac{\psi^{j}_{t}(y) + \phi^{j}_{t}(X) - c(X,y)}{\lambda}\right)
    \mu_{t}(dx)
  \right]
  \\
  &=
  \frac{1}{\tau}\min(\lambda, \tau)
  \log
  \sum_{j=1}^{k}
  w_{j}
  \left[
    \int_{\mathcal{X}}
    \Delta_{t}^{j}(y)\nu^{j}(dy)
    \frac{d\nu^{j}_{\vecpsi_{t}}}{d\nu^{j}}(y)
  \right]
  \\
  &=
  \frac{1}{\tau}\min(\lambda, \tau)
  \log
  \sum_{j=1}^{k}
  w_{j}
  \mathbf{E}_{Y \sim \nu^{j}_{\vecpsi_{t}}}
  \left[
    \Delta^{j}_{t}(y)
  \right].
\end{align}
This completes the proof of Lemma~\ref{lemma:log-Z-ratio-bound}
when $\tau \geq \lambda$.

\paragraph{The case \texorpdfstring{$\tau < \lambda$}{tau < lambda}.}

For $j \in \{1, \dots, k\}$ and any $x \in \mathcal{X}$ define the measure
$\rho_{x}$ by
\begin{equation}
  \rho^{j}_{x}(dy) =
  \nu^{j}(dy)
  \exp\left(\frac{\psi^{j}_{t}(y) + \phi^{j}_{t}(x) -
  c(x,y)}{\lambda}\right).
\end{equation}
By the definition of $\phi^{j}_{t}$, we have
\begin{align}
  &\int_{\mathcal{X}} \rho^{j}_{x}(dy)
  \\
  &=
  \int_{\mathcal{X}}
    \nu(dy)
    \exp\left(\frac{\psi^{j}_{t}(y) - c(x,y)}{\lambda}\right)
    \exp\left(\frac{\phi^{j}_{t}(x)}{\lambda}\right)
  \\
  &=
  \int_{\mathcal{X}}
    \nu(dy)
    \exp\left(\frac{\psi^{j}_{t}(y) - c(x,y)}{\lambda}\right)
    \exp\left(-\log\int_{\mathcal{X}}
      \nu^{j}(dy')
      \exp\left(\frac{\psi^{j}_{t}(y') - c(x,y')}{\lambda}\right)
    \right)
  \\
  &=
  1
\end{align}
In particular, $\rho_{x}$ is a probability measure.
Hence,
\eqref{eq:Zt-ratio-bound-before-case-splitting} can be rewritten as
\begin{align}
  \log \frac{Z_{\vecpsi_{t+1}}}{Z_{\vecpsi_{t}}}
  &\leq
  \log
  \sum_{j=1}^{k}
  w_{j}
  \mathbf{E}_{X \sim \mu_{t}}\left[
    \mathbf{E}_{Y \sim \rho^{j}_{X}}\left[
    \Delta^{j}_{t}(Y)^{\eta}
    \bigg\vert X
  \right]^{\lambda/\tau}
\right]
\end{align}
Because $\lambda/\tau > 1$, the function $x \mapsto x^{\lambda/\tau}$ is
convex. Applying Jensen's inequality to the conditional expectation and using
the fact that $\eta\lambda/\tau = 1$, it follows that
\begin{align}
  \log \frac{Z_{\vecpsi_{t+1}}}{Z_{\vecpsi_{t}}}
  &\leq
  \log
  \sum_{j=1}^{k}
  w_{j}
  \mathbf{E}_{X \sim \mu_{t}}\left[
    \mathbf{E}_{Y \sim \rho^{j}_{X}}\left[
    \Delta^{j}_{t}(Y)
    \bigg\vert X
  \right]
  \right]
  \\
  &=
  \log
  \sum_{j=1}^{k}
  w_{j}
  \int_{\mathcal{X}}
    \mu_{t}(dx)
  \int_{\mathcal{X}}
  \Delta_{t}^{j}(y)
  \exp\left(\frac{\psi^{j}_{t}(y) + \phi^{j}_{t}(x) -
    c(x,y)}{\lambda}\right)
    \nu(dy).
    \label{eq:proof-of-one-step-improvement-proposition-pre-Fubini-step}
\end{align}
By the definition of $\nu^{j}_{\vecpsi_{t}}$ we have
\begin{equation}
  \frac{d\nu^{j}_{\vecpsi_{t}}}{d\nu^{j}}(y)
  =
  \int_{\mathcal{X}}
  \exp\left(\frac{\psi^{j}_{t}(y) + \phi^{j}_{t}(x) -
  c(x,y)}{\lambda}\right)\mu_{t}(dx).
\end{equation}
Interchanging the order of integration in
\eqref{eq:proof-of-one-step-improvement-proposition-pre-Fubini-step}
and plugging in the above equation yields
\begin{align}
  \log \frac{Z_{\vecpsi_{t+1}}}{Z_{\vecpsi_{t}}}
  &\leq
  \log
  \sum_{j=1}^{k}
  w_{j}
  \int_{\mathcal{X}}
  \left[
    \int_{\mathcal{X}}
    \exp\left(\frac{\psi^{j}_{t}(y) + \phi^{j}_{t}(x) -
    c(x,y)}{\lambda}\right)
    \mu_{t}(dx)
  \right]
  \Delta^{j}_{t}(y)
  \nu^{j}(dy)
  \\
  &=
  \log
  \sum_{j=1}^{k}
  w_{j}
  \int_{\mathcal{X}}
  \left[
    \frac{d\nu^{j}_{\vecpsi_{t}}}{d\nu^{j}}(y)
  \right]
  \Delta_{t}^{j}(y)
  \nu^{j}(dy)
  \\
  &=
  \log
  \sum_{j=1}^{k}
  w_{j}
  \mathbf{E}_{Y \sim \nu^{j}_{\vecpsi_{t}}}\left[
    \Delta^{j}_{t}(Y)
  \right].
\end{align}
This completes the proof of Lemma~\ref{lemma:log-Z-ratio-bound}.
\hfill\qed

\subsection{Proof of Lemma~\ref{lemma:zt-ratio}}
\label{sec:proof-of-Zt-ratio-lemma}
To simplify the notation, denote $Z_{t} = Z_{\vecpsi_{t}}$.
Let $x \in \mathcal{X}$ and $t \geq 0$. We have $\mu_{t} \ll \mu_{t+1}$ with
the Radon-Nikodym derivative $d\mu_{t+1}/d\mu_{t}$ given by
\begin{align}
  \frac{d\mu_{t+1}}{d\mu_{t}}(x)
  &= \frac{Z_{t}}{Z_{t+1}} \exp\left(
    \frac{-\sum_{j=1}^{k}w_{k}(\phi^{j}_{t+1}(x) - \phi^{j}_{t}(x))}{\tau}
  \right)
  \\
  &=
  \frac{Z_{t}}{Z_{t+1}} \prod_{j=1}^{k}
    \exp\left(
      \frac{-\phi^{j}_{t+1}(x) + \phi^{j}_{t}(x)}{\tau}
    \right)^{w_{j}}.
\end{align}
Multiplying both sides by $Z_{t+1}/Z_{t}$ and taking expectations with respect
to $\mu_{t}$ yields
\begin{align}
  \frac{Z_{t+1}}{Z_{t}}
  &=
  \mathbf{E}_{X \sim \mu_{t+1}}\left[
    \frac{Z_{t+1}}{Z_{t}}
  \right]
  \\
  &=
  \mathbf{E}_{X \sim \mu_{t}}\left[
    \frac{Z_{t+1}}{Z_{t}}
    \frac{d\mu_{t+1}}{d\mu_{t}}(X)
  \right]
  \\
  &=
  \mathbf{E}_{X \sim \mu_{t}}\left[
    \prod_{j=1}^{k}
    \exp\left(
      \frac{-\phi^{j}_{t+1}(X) + \phi^{j}_{t}(X)}{\tau}
    \right)^{w_{j}}
  \right].
\end{align}
In the case $\tau < \lambda$, the proof is complete by the Artihmetic-Geometric
mean inequality (recall that $w_{j} > 0$ for $j=1,\dots,k$ and
$\sum_{j=1}^{k}w_{j} = 1$).
On the other hand, if $\tau \geq \lambda$ then $x \mapsto x^{\lambda/\tau}$ is
concave. Hence, it follows that
\begin{align}
  \log
  \frac{Z_{t+1}}{Z_{t}}
  &=
  \log
  \mathbf{E}_{X \sim \mu_{t}}\left[
    \left(
    \prod_{j=1}^{k}
    \exp\left(
      \frac{-\phi^{j}_{t+1}(X) + \phi^{j}_{t}(X)}{\tau}
    \right)^{w_{j}\tau/\lambda}
    \right)^{\lambda/\tau}
  \right]
  \\
  &\leq
  \log
  \mathbf{E}_{X \sim \mu_{t}}\left[
    \left(
    \prod_{j=1}^{k}
    \exp\left(
      \frac{-\phi^{j}_{t+1}(X) + \phi^{j}_{t}(X)}{\tau}
    \right)^{w_{j}\tau/\lambda}
    \right)
  \right]^{\lambda/\tau}
  \\
  &=
  \frac{\lambda}{\tau}
  \log
  \mathbf{E}_{X \sim \mu_{t}}\left[
    \left(
    \prod_{j=1}^{k}
    \exp\left(
      \frac{-\phi^{j}_{t+1}(X) + \phi^{j}_{t}(X)}{\tau}
    \right)^{w_{j}\tau/\lambda}
    \right)
  \right]
  \\
  &\leq
  \sum_{j=1}^{k}
  w_{j}
  \mathbf{E}_{X \sim \mu_{t}}\left[
    \exp\left(
      \frac{-\phi^{j}_{t+1}(X) + \phi^{j}_{t}(X)}{\tau}
    \right)^{\tau/\lambda}
  \right],
\end{align}
where the final step follows via the Arithmetic-Geometric mean
inequality. This completes the proof of Lemma~\ref{lemma:zt-ratio}.
\hfill\qed

\section{Proof of Theorem~\ref{thm:inexact-scheme-convergence}}
\label{sec:inexact-algorithm-convergence-proof}
For every $t \geq 0$ and $j \in \{1,\dots,k\}$, let
$\widetilde{\nu}^{j}_{t}$
be the distribution returned by the approximate Sinkhorn oracle that satisfies
the properties listed in Definition~\ref{dfn:approximate-sinkhorn-oracle}.
We follow along the lines of proof of
Theorem~\ref{thm:exact-scheme-convergence}.

First, we will establish an upper
bound on the oscillation norm of the iterates $\widetilde{\vecpsi}_{t}$.
Indeed, by the property four in
Definition~\ref{dfn:approximate-sinkhorn-oracle} we have
\begin{equation}
  \|\widetilde{\psi}^{j}_{t+1}\|_{\mathrm{osc}}
  \leq
  (1-\eta)\|\widetilde{\psi}^{j}_{t}\|_{\mathrm{osc}} + \eta
  c_{\infty}(\mathcal{X}).
\end{equation}
Since $\widetilde{\psi}^{j}_{0} = 0$, for any $t \geq 0$ we have
$\|\widetilde{\psi}^{j}_{t}\|_{\mathrm{osc}} \leq c_{\infty}(\mathcal{X})$.

Let $\tilde{\delta}_{t} = E_{\lambda,\tau}^{\vecnu, w}(\psi^{*}) -
E_{\lambda,\tau}^{\vecnu, w}(\widetilde{\vecpsi}_{t})$ be the suboptimality gap
at time $t$. Using the concavity upper bound
\eqref{eq:concave-suboptimality-gap} and the property two in
Definition~\ref{dfn:approximate-sinkhorn-oracle} we have
\begin{align}
  \tilde{\delta}_{t}
  &\leq
  2c_{\infty}(\mathcal{X})\sum_{j=1}^{k}w_{j}\|\nu^{j} -
  \nu_{t}^{j}\|_{\mathrm{TV}}
  \\
  &\leq
  \varepsilon +
  2c_{\infty}(\mathcal{X})\sum_{j=1}^{k}w_{j}\|\nu^{j} -
  \widetilde{\nu}_{t}^{j}\|_{\mathrm{TV}}
  \\
  &\leq
  \varepsilon +
  \sqrt{2}c_{\infty}(\mathcal{X})\sum_{j=1}^{k}w_{j}
  \sqrt{\kl{\nu^{j}}{\widetilde{\nu}_{t}^{j}}}
  \\
  &\leq
  \varepsilon +
  \sqrt{2}c_{\infty}(\mathcal{X})
  \sqrt{\sum_{j=1}^{k}w_{j}\kl{\nu^{j}}{\widetilde{\nu}_{t}^{j}}}.
\end{align}
Combining the property three
stated in the
Definition~\ref{dfn:approximate-sinkhorn-oracle} with
Lemma~\ref{lemma:log-Z-ratio-bound} we obtain
\begin{align}
  \tilde{\delta}_{t} - \tilde{\delta}_{t+1}
  &\geq \min(\lambda,\tau)
  \sum_{j=1}^{k}w_{j}\kl{v^{j}}{\widetilde{v}_{t}^{j}}
  -\min(\lambda,\tau)
  \log\left(
    \sum_{j=1}^{k}
    w_{j}
    \int_{\mathcal{X}}
    \frac{d\nu_{t}}{d\widetilde{\nu}_{t}}(y)
    \nu^{j}(dy)
  \right)
  \\
  &\geq
  \sum_{j=1}^{k}w_{j}\kl{v^{j}}{\widetilde{v}_{t}^{j}}
  -\min(\lambda,\tau)
  \log\left(
    1 + \varepsilon^{2}/(2c_{\infty}(\mathcal{X})^{2})
  \right)
  \\
  &\geq
  \min(\lambda,\tau)
  \sum_{j=1}^{k}w_{j}\kl{v^{j}}{\widetilde{v}_{t}^{j}}
  -\frac{\min(\lambda,\tau)}{2c_{\infty}(\mathcal{X})^{2}}
  \varepsilon^{2}
  \\
  &\geq
  \frac{\min(\lambda,\tau)}{2c_{\infty}(\mathcal{X})^{2}}
  \max\left\{0, \tilde{\delta}_{t} - \varepsilon \right\}^{2}
  -\frac{\min(\lambda,\tau)}{2c_{\infty}(\mathcal{X})^{2}}
  \varepsilon^{2}.
\end{align}
Provided that $\widetilde{\delta}_{t} \geq 2\varepsilon$ it holds that
\begin{equation}
  (\tilde{\delta}_{t} - 2\varepsilon) - (\tilde{\delta}_{t+1} - 2\varepsilon)
  \geq
  \frac{\min(\lambda,\tau)}{2c_{\infty}(\mathcal{X})}
  (\tilde{\delta}_{t} - 2\varepsilon)^{2}.
\end{equation}
Let $T$ be the first index such that $\widetilde{\delta}_{T+1} < 2\varepsilon$
and set $T = \infty$ if no such index exists.
Then, the above equation is valid for any $t \leq T$. In particular, repeating
the proof of Theorem~\ref{thm:exact-scheme-convergence}, for any
$t \leq T$ we have
\begin{equation}
  \widetilde{\delta}_{t} - 2\varepsilon \leq
  \frac{2c_{\infty}(\mathcal{X})^{2}}{\min(\lambda,\tau)}\frac{1}{t},
\end{equation}
which completes the proof of this theorem. \hfill\qed

\section{Proof of Lemma~\ref{lemma:approximate-oracle-implementation}}
\label{sec:proof-of-approximate-oracle-implementation}

  The first property -- the positivity of the probability mass function of
  $\widetilde{\nu}^{j}$ -- is immediate from its definition.

  To simplify the notation, denote in what follows
  \begin{equation}
    K^{j}(x,y) = \exp\left(
    \frac{\phi_{\psi^{j}}(x) + \psi^{j}(y) - c(x,y)}{\lambda}\right).
  \end{equation}
  With this notation, recall that
  $$
    \widehat{\nu}^{j}_{\vecpsi}(y^{j}_{l})
    = \frac{1}{n}\sum_{i=1}^{n}\nu^{j}(y^{j}_{l})K(X_{i},y^{j}_{l}).
  $$
  The above is a sum of $n$ non-negative random variables bounded by one
  with expectation
  $$
    (\nu' )^{j}(y^{j}_{l})
    = \mathbf{E}_{X \sim \mu_{\vecpsi}'}\left[
      \nu^{j}(y^{j}_{l})
    \right]
  $$
  It follows by Hoeffding's inequality
  and the union bound that with probability
  at least $1-\delta$ the following holds for any
  $j \in \{1,\dots,k\}$ and any $l \in \{1, \dots, m_{j}\}$:
  \begin{equation}
    \label{eq:hoeffding}
    \left|
      \widehat{\nu}_{\vecpsi}(y^{j}_{l})
      - (\nu' )^{j}(y^{j}_{l})
    \right|
    \leq \sqrt{\frac{2\log\left(\frac{2m}{\delta}\right)}{n}}.
  \end{equation}
  In particular, the above implies that
  \begin{align}
    \|\widetilde{\nu}^{j}_{\vecpsi} - \nu^{j}_{\vecpsi}\|_{\mathrm{TV}}
    &\leq 2\zeta +
    (1-\zeta)
    \|\widetilde{\nu}^{j}_{\vecpsi} - \nu^{j}_{\vecpsi}\|_{\mathrm{TV}}
    \\
    &\leq 2\zeta +
    (1-\zeta)
    \|\widetilde{\nu}^{j}_{\vecpsi} - (\nu')^{j}\|_{\mathrm{TV}}
    + (1-\zeta)\|(\nu')^{j} - \nu^{j}_{\vecpsi}\|_{\mathrm{TV}}
    \\
    &\leq 2\zeta +
    \|\widetilde{\nu}^{j}_{\vecpsi} - (\nu')^{j}\|_{\mathrm{TV}}
    + \|(\nu')^{j} - \nu^{j}_{\vecpsi}\|_{\mathrm{TV}}
    \\
    &\leq 2\zeta + m_{j}\varepsilon_{\mu}
    + m_{j}
    \sqrt{\frac{2\log\left(\frac{2m}{\delta}\right)}{n}}.
  \end{align}
  Notice that the above bound can be made arbitrarily close to
  $m_{j}\varepsilon_{\mu}$ by taking a large enough $n$ and a small enough
  $\zeta$.
  This proves the second property of
  Definition~\ref{dfn:approximate-sinkhorn-oracle}.

  To prove the third property,
  observe that
  \begin{align}
    \mathbf{E}_{Y \sim \nu^{j}}\left[
      \frac{\nu^{j}_{\vecpsi}(Y)}{\widetilde{\nu}^{j}_{\vecpsi}(Y)}
    \right]
    &=
    \mathbf{E}_{Y \sim \nu^{j}}\left[
      \frac{\widehat{\nu}^{j}_{\vecpsi}(Y)}
      {\widetilde{\nu}^{j}_{\vecpsi}(Y)}
      + \frac{\nu^{j}_{\vecpsi}(Y) - \widehat{\nu}^{j}_{\vecpsi}(Y)
      }{\widetilde{\nu}^{j}_{\vecpsi}(Y)}
    \right]
    \\
    &\leq
    \mathbf{E}_{Y \sim \nu^{j}}\left[
      \frac{1}{1-\zeta} + \frac{\nu^{j}_{\vecpsi}(Y) - \widehat{\nu}^{j}_{\vecpsi}(Y)
      }{\widetilde{\nu}^{j}_{\vecpsi}(Y)}
    \right]
    \\
    &\leq
    \mathbf{E}_{Y \sim \nu^{j}}\left[
      1
      + \frac{\zeta}{1-\zeta}
      + \frac{\left|\nu^{j}_{\vecpsi}(Y) -
        \widehat{\nu}^{j}_{\vecpsi}(Y)\right|
      }{\widetilde{\nu}^{j}_{\vecpsi}(Y)}
    \right]
    \\
    &\leq
    1 +
    \frac{\zeta}{1-\zeta}
    + \frac{1}{\zeta}
    \|\nu^{j}_{\vecpsi}(Y) -
    \widehat{\nu}^{j}_{\vecpsi}(Y)\|_{\mathrm{TV}}
    \\
    &\leq
    1
    +
    2\zeta
    + \frac{1}{\zeta}
    \left(m_{j}\varepsilon_{\mu} + m_{j}
      \sqrt{\frac{2\log\left(\frac{2m}{\delta}\right)}{n}}
    \right).
  \end{align}
  This concludes the proof of the third property.

  It remains to prove the fourth property of
  Definition~\ref{dfn:approximate-sinkhorn-oracle}.
  Observe that for any $y,y'$ we have
  \begin{align}
    &\left(
      \psi^{j}(y) - \eta\lambda\log\frac{\widetilde{\nu}^{j}(y)}{\nu^{j}(y)}
    \right)
    -
    \left(
      \psi^{j}(y') - \eta\lambda\log\frac{\widetilde{\nu}^{j}(y')}{\nu^{j}(y')}
    \right)
    \\
    &=
    \left(\psi^{j}(y) - \psi^{j}(y')\right)
    + \eta\lambda\log
    \left(
    \frac
    {
      \zeta + (1-\zeta)\frac{1}{n}\sum_{i=1}^{n}K^{j}(X_{i}, y')
    }
    {
      \zeta + (1-\zeta)\frac{1}{n}\sum_{i=1}^{n}K^{j}(X_{i}, y)
    }
    \right)
    \\
    &=
    \left(\psi^{j}(y) - \psi^{j}(y')\right)
    + \eta\lambda\log
    \left(
    \frac
    {
      \frac{\zeta}{1-\zeta} + \frac{1}{n}\sum_{i=1}^{n}K^{j}(X_{i}, y')
    }
    {
      \frac{\zeta}{1-\zeta} + \frac{1}{n}\sum_{i=1}^{n}K^{j}(X_{i}, y)
    }
    \right)
    \\
    &\leq
    \left(\psi^{j}(y) - \psi^{j}(y')\right)
    + \eta\lambda\log
    \left(
    \frac
    {
      \frac{\zeta}{1-\zeta} + \exp\left(\frac{c_{\infty}(\mathcal{X}) +
        \psi^{j}(y') - \psi^{j}(y)}{\lambda}
      \right)\frac{1}{n}\sum_{i=1}^{n}K^{j}(X_{i}, y)
    }
    {
      \frac{\zeta}{1-\zeta} + \frac{1}{n}\sum_{i=1}^{n}K^{j}(X_{i}, y)
    }
    \right).
    \label{eq:psi-oscilation-zeta-bound-intermediate}
  \end{align}
  Now observe that for any $a,b > 0$ the function $g:[0,\infty) \to (0,\infty)$
  defined by $g(x) = (x+a)/(x+b)$ is increasing if $a < b$ and decreasing if $a
  \geq b$. Thus, $g$ is maximized either at zero or at infinity.
  It thus follows that
  \begin{align}
    &\eta\lambda\log\left(
      \frac
    {
      \frac{\zeta}{1-\zeta}
      +
      \exp\left(\frac{c_{\infty}(\mathcal{X})}{\lambda}
      \right)
      \frac{1}{n}\sum_{i=1}^{n}K^{j}(X_{i}, y)
    }
    {
      \frac{\zeta}{1-\zeta} + \frac{1}{n}\sum_{i=1}^{n}K^{j}(X_{i}, y)
    }
    \right)
    \\
    &\leq
    \begin{cases}
      \eta c_{\infty}(\mathcal{X})
      -\eta(\psi^{j}(y) - \psi^{j}(y'))
      &\text{if }
      \exp\left(\frac{c_{\infty}(\mathcal{X})}{\lambda}
      \right) \geq 1
      \\
      0&\text{otherwise.}
    \end{cases}
  \end{align}
  This proves the claim and completes the proof of this lemma.
  \hfill\qed

\section{Proof of Theorem~\ref{thm:inexact-algorithm-implementation}}
\label{sec:langevin-sampling-guarantees}
The purpose of this section is to show how sampling via Langevin Monte Carlo
algorithm yields the first provable convergence guarantees for computing
barycenters in the free-support setup (cf.\ the discussion at the end of
Section~\ref{sec:doubly-entropic-barycenters}). In particular, we provide
computational guarantees for implementing Algorithm~\ref{alg:inexact}.

A measure $\mu$ is said to satisfy the logarithmic Sobolev inequality (LSI)
with constant $C$ if for all sufficiently smooth functions $f$ it holds that
\begin{equation}
  \mathbf{E}_{\mu}[f^{2}\log f^{2}] - \E_{\mu}[f^{2}]\log\E_{\mu}[g^{2}]
  \leq 2C\mathbf{E}_{\mu}[\|\nabla f\|^{2}_{2}].
\end{equation}
To sample from a measure $\mu(dx) = \exp(-f(x))dx$ supported on $\R^{d}$, the
unadjusted Langevin Monte Carlo algorithm is defined via the following
recursive update rule:
\begin{equation}
  \label{eq:langevin-monte-carlo-updates}
  x_{k+1} = x_{k} - \eta \nabla f(x_{k}) + \sqrt{2\eta}
  Z_{k},\quad\text{where}\quad Z_{k} \sim \mathcal{N}(0, I_{d}).
\end{equation}
The following Theorem is due to \citet*[Theorem 3]{vempala2019rapid}.
\begin{theorem}
  \label{thm:langevin}
  Let $\mu(dx) = \exp(-f(x))dx$ be a measure on $\R^{d}$. Suppose that $\mu$
  satisfies LSI a with constant $C$ and that $f$
  has $L$-Lipschitz gradient with respect to the Euclidean norm.
  Consider the sequence of iterates $(x_{k})_{k \geq 0}$ defined via
  \eqref{eq:langevin-monte-carlo-updates} and let
  let $\rho_{k}$ be the distribution of $x_{k}$.
  Then, for any $\varepsilon > 0$, any $\eta \leq
  \frac{1}{8L^{2}C}\min\{1,\frac{\varepsilon}{4d}\}$, and any
  $k \geq \frac{2C}{\eta}\log\frac{2\kl{\rho_{0}}{\mu}}{\varepsilon}$,
  it holds that
  \begin{equation}
    \kl{\rho_{k}}{\mu} \leq \varepsilon.
  \end{equation}
\end{theorem}
Thus, LSI on the measure $\mu$ provides convergence guarantees on
$\kl{\rho_{k}}{\mu}$. It is shown in \cite[Lemma 1]{vempala2019rapid} how to
initialize the iterate $x_{0}$ so that $\kl{\rho_{0}}{\mu}$ scales linearly
with the ambient dimension $d$ up to some additional terms.
The final condition described in Problem
Setting~\ref{setup:ball-and-squared-loss}
ensures that (by \cite[Lemma 1]{vempala2019rapid})
for any $\sigma > 0$, the initialization
scheme $x_{0} \sim \mathcal{N}(x_{\vecpsi}, I_{d})$
for the  Langevin algorithm \eqref{eq:langevin-monte-carlo-updates}
satisfies
\begin{equation}
  \label{eq:initialization-langevin-KL-bound}
  \kl{\rho_{0}}{\mu_{\vecpsi,\sigma}}
  \leq \frac{c_{\infty}(\mathcal{X})}{\tau} +
  \frac{d}{2}\log\frac{L_{\sigma}}{2\pi},
\end{equation}
where $L_{\sigma}$ is the smoothness constant of $V_{\vecpsi}/\tau +
\mathrm{dist}(x, \mathcal{X})/(2\sigma^{2})$ (see
Lemma~\ref{lemma:properties-of-mu-sigma}) and
$\mu_{\vecpsi,\sigma}$ is the probability measure defined in
\eqref{eq:mu-sigma-dfn}.

To implement the approximate Sinkhorn oracle described in
Definition~\ref{dfn:approximate-sinkhorn-oracle}, we can combine
Lemma~\ref{lemma:approximate-oracle-implementation} with approximate sampling
via Langevin Monte Carlo; note that by Pinsker's inequality, Kullback-Leibler
divergence guarantees provide total variation guarantees which are sufficient
for the application of Lemma~\ref{lemma:approximate-oracle-implementation}.
Therefore, providing provable convergence guarantees for
Algorithm~\ref{alg:inexact}
amounts to proving that we can do arbitrarily accurate approximate sampling
from distributions of the form
\begin{equation}
  \mu_{\vecpsi}(dx)
  \propto \mathbb{1}_{\mathcal{X}}(x)\exp(-V_{\vecpsi}(x)/\tau)dx,
  \quad\text{where}\quad V_{\vecpsi}(x) =
  \sum_{j=1}^{k}w_{j}\phi^{j}_{\psi^{j}}(x).
\end{equation}
Here $\mathbb{1}_{\mathcal{X}}$ is the indicator function of $\mathcal{X}$,
$\vecpsi$ is an arbitrary iterate generated by
Algorithm~\ref{alg:inexact}, and we consider the free-support setup
characterized via the choice $\piref(dx) = \mathbb{1}_{\mathcal{X}}dx$.

Notice that we cannot apply Theorem~\ref{thm:langevin} directly because the
measure $\mu_{\vecpsi}$ defined above has constrained support while
Theorem~\ref{thm:langevin} only applies for measures supported on all of
$\mathbb{R}^{d}$. Nevertheless, we will show that the compactly supported
measure $\mu_{\vecpsi}$ can be approximated by a measure $\mu_{\vecpsi,
\sigma}$, where the parameter $\sigma$ will trade-off LSI constant of
$\mu_{\vecpsi, \sigma}$ against the total variation norm between the two
measures. To this end, define
\begin{equation}
  \label{eq:mu-sigma-dfn}
  \mu_{\vecpsi, \sigma} =
  \propto
\exp(-V_{\vecpsi}(x)/\tau
-\mathrm{dist}(x, \mathcal{X})^{2}/(2\sigma^{2}))dx,
  \quad\text{where}\quad
  \mathrm{dist}(x,\mathcal{X})
  = \inf_{y \in \mathcal{X}}\|x - y\|_{2}.
\end{equation}

The following lemma, proved at the end of this section,
collects the main properties of the measure $\mu_{\vecpsi,\sigma}$.
\begin{lemma}
  \label{lemma:properties-of-mu-sigma}
  Consider the setup described in Problem
  Setting~\ref{setup:ball-and-squared-loss}.
  Let $\vecpsi$ be any iterate generated by Algorithm~\ref{alg:inexact} and
  let $\mu_{\vecpsi, \sigma}$ be the distribution defined in
  \eqref{eq:mu-sigma-dfn}. Then, the measure $\mu_{\vecpsi,\sigma}$ satisfies
  the following properties:
  \begin{enumerate}
    \item For any $\sigma \in (0,1/4]$ it holds that
     \begin{equation}
       \|\mu_{\vecpsi} - \mu_{\vecpsi, \sigma}\|_{\mathrm{TV}}
        \leq
        2\sigma
        \exp\left(\frac{8R^{2}}{\tau}\right)
        \left[
          \left(4Rd^{-1/4}\right)^{d-1}
          + 1
        \right].
    \end{equation}

   \item
     Let $V_{\sigma}(x) = \exp(-V_{\vecpsi}(x)/\tau
     -\mathrm{dist}(x, \mathcal{X})^{2}/(2\sigma^{2}))$; thus $\mu_{\vecpsi,
     \sigma}(dx) = \exp(-V_{\sigma}(x))dx$.
     The function $V_{\sigma}$ has $L_{\sigma}$-Lipschitz gradient where
       \begin{equation}
         \label{eq:L-sigma-dfn}
         L_{\sigma} =
         \frac{1}{\tau} + \frac{1}{\tau\lambda}4R^{2}\max_{j}m_{j} +
         \frac{1}{\sigma^{2}}.
       \end{equation}

     \item The measure $\mu_{\vecpsi,\sigma}$ satisfies LSI with a constant
       $C_{\sigma} = \mathrm{poly}(R, \exp(R^{2}/\tau), L_{\sigma})$.
  \end{enumerate}
\end{lemma}
Above, the notation $C = \mathrm{poly}(x,y,z)$ denotes a constant that depends
polynomially on $x,y$ and $z$.
With the above lemma at hand, we are ready to prove
Theorem~\ref{thm:inexact-algorithm-implementation}.

\begin{proof}[Proof of Theorem~\ref{thm:inexact-algorithm-implementation}]
  Let $\vecpsi$ be an arbitrary iterate generated via
  Algorithm~\ref{alg:inexact}.
  We can simulate a step of approximate Sinkhorn oracle with accuracy
  $\varepsilon$ via Lemma~\ref{lemma:approximate-oracle-implementation}
  (with $\zeta = \varepsilon/4$)
  in time $\mathrm{poly}(n,m,d)$ provided
  access to $n = \mathrm{poly}(\varepsilon^{-1}, m, \log(m/\delta))$ samples
  from any distribution $\mu_{\vecpsi}'$ such that
  \begin{equation}
    \label{eq:mu-prime-needed-approximation}
    \|\mu_{\vecpsi}' - \mu_{\vecpsi}\|_{\mathrm{TV}} \leq
    \frac{\varepsilon^{2}}{16m}.
  \end{equation}
  To find a choice of $\mu_{\vecpsi}'$ satisfying the above bound, consider
  the distribution
  \begin{equation}
    \mu_{\vecpsi, \sigma}\quad\text{with}\quad
      \sigma = \frac{\varepsilon^{2}}{32m}\cdot
      \left(
        2
        \exp\left(\frac{8R^{2}}{\tau}\right)
        \left[
          \left(4Rd^{-1/4}\right)^{d-1}
          + 1
        \right]
      \right)^{-1}.
  \end{equation}
  Let $C_{\sigma}$ and $L_{\sigma}$ be the LSI and smoothness constants of the
  distribution $\mu_{\vecpsi,\sigma}$ provided in
  Lemma~\ref{lemma:properties-of-mu-sigma}.
  By Theorem~\ref{thm:langevin}, it suffices to run the Langevin algorithm
  ~\eqref{eq:langevin-monte-carlo-updates} for $\mathrm{poly}(\varepsilon^{-1},
  m, d, C_{\sigma}, L_{\sigma})$ number of iterations to obtain a sample from
  a distribution $\widetilde{\mu}_{\vecpsi,\sigma}$ such that
  \begin{equation}
    \|\widetilde{\mu}_{\vecpsi, \sigma} - \mu_{\vecpsi, \sigma}\|_{\mathrm{TV}}
    \leq \frac{\varepsilon^{2}}{32m}.
  \end{equation}
  In particular, by the triangle inequality for the total variation norm,
  the choice $\mu_{\vecpsi}' = \widetilde{\mu}_{\vecpsi, \sigma}$ satisfies
  \eqref{eq:mu-prime-needed-approximation}. This finishes the proof.
\end{proof}

\subsection{Proof of Lemma~\ref{lemma:properties-of-mu-sigma}}
\label{sec:proof-of-lemma-properties-of-mu-sigma}
To simplify the notation, denote $\mu = \mu_{\vecpsi}, \mu_{\sigma} =
\mu_{\vecpsi, \sigma}$, $V(x) = V_{\vecpsi}(x)/\tau$,
and
$V_{\sigma}(x) = V(x)/\tau
  +\mathrm{dist}(x, \mathcal{X})^{2}/(2\sigma^{2})$.

\paragraph{Total variation norm bound.}

With the above shorthand notation, we have
\begin{equation}
  \mu(dx) = \mathbb{1}_{\mathcal{X}}Z^{-1}\exp(-V(x))dx,\quad\text{where}\quad
  Z = \int_{\mathcal{X}}\exp(-V(x))dx
\end{equation}
and
$$
  \mu_{\sigma}(dx) = (Z+Z_{\sigma})^{-1}\exp(-V_{\sigma}(x))dx,
  \quad\text{where}\quad
  Z_{\sigma} = \int_{\mathbb{R}^{d} \backslash
  \mathcal{X}}\exp(-V_{\sigma}(x))dx.
$$
We have
\begin{align}
  \|\mu - \mu_{\sigma}\|_{\mathrm{TV}}
  &=
  \int_{\R^{d}\backslash\mathcal{X}}(Z+Z_{\sigma})^{-1}\exp(-V_{\sigma}(x))dx
  + \int_{\mathcal{X}}|(Z+Z_{\sigma})^{-1} - Z^{-1}|\exp(-V(x))dx
  \\
  &=
  \frac{2Z_{\sigma}}{Z + Z_{\sigma}}
  \leq
  \frac{2Z_{\sigma}}{Z}
  \leq
  2\exp\left(\frac{c_{\infty}(\mathcal{X})}{\tau}\right)Z_{\sigma}
  \leq
  2\exp\left(\frac{4R^{2}}{\tau}\right)Z_{\sigma}.
\end{align}
We thus need to upper bound $Z_{\sigma}$.
Let $\mathrm{Vol}(A)$ be the Lebesgue measure of the set $A$,
let $\partial A$ denote the boundary of $A$,
and let $A + B = \{a + b : a \in A, b \in B\}$ be the Minkowski sum of sets
$A$ and $B$. Using the facts that for each $j \in \{1,\dots,k\}$ we have
$\sup_{y \in \mathcal{X}} \psi^{j}(y) \leq c_{\infty}(\mathcal{X}) \leq 4R^{2}$ and
that $\mathcal{X} \subseteq \mathcal{B}_{R}
= \{x : \|x\|_{2} \leq R\}$ we have
\begin{align}
  Z_{\sigma}
  &= \int_{\mathbb{R}^{d} \backslash
  \mathcal{X}}\exp(-V_{\sigma}(x))dx
  \\
  &\leq
  \exp\left(\frac{4R^{2}}{\tau}\right) \int_{\mathbb{R}^{d} \backslash
    \mathcal{X}}\exp\left(
  -\frac{\mathrm{dist}(x, \mathcal{X})}{2\sigma^{2}}\right)dx
  \\
  &=
  \exp\left(\frac{4R^{2}}{\tau}\right)
  \int_{0}^{\infty}
  \mathrm{Vol}(\partial(\mathcal{X} + \mathcal{B}_{x}))
    \exp\left(
    -\frac{x^{2}}{2\sigma^{2}}\right)dx
  \\
  &\leq
  \exp\left(\frac{4R^{2}}{\tau}\right)
  \int_{0}^{\infty}
  \mathrm{Vol}(\partial \mathcal{B}_{R+x})
    \exp\left(
    -\frac{x^{2}}{2\sigma^{2}}\right)dx
  \\
  &=
  \exp\left(\frac{4R^{2}}{\tau}\right)
  \frac{\pi^{d/2}}{\Gamma(d/2)}
  \int_{0}^{\infty}
    (R + x)^{d-1}
    \exp\left(
    -\frac{x^{2}}{2\sigma^{2}}\right)dx.
\end{align}
Bounding $(R+x)^{d-1} \leq 2^{d-1}R^{d-1} + 2^{d-1}x^{d-1}$
and computing the integrals results in
\begin{align}
  \|\mu - \mu_{\sigma}\|_{\mathrm{TV}}
  &\leq
  2\exp\left(\frac{8R^{2}}{\tau}\right)
  \frac{\pi^{d/2}}{\Gamma(d/2)}
  2^{d-1}
  \left[
    R^{d-1}
    \sigma
    \frac{\sqrt{\pi}}{2}
    +
    2^{d/2-1}\Gamma(d/2)\sigma^{d}
  \right]
  \\
  &\leq
  2\sigma\exp\left(\frac{8R^{2}}{\tau}\right)
  \left[
    \frac{(2R)^{d-1}}{\Gamma(d/2)}
    + (4\sigma)^{d-1}
  \right].
\end{align}
Using the assumption $\sigma \leq 1/4$ and using the bound $\Gamma(d) \geq
(d/2)^{d/2}$ we can further simplify the above bound to
\begin{equation}
  \|\mu - \mu_{\sigma}\|_{\mathrm{TV}}
  \leq
  2\sigma\exp\left(\frac{8R^{2}}{\tau}\right)
  \left[
    \left(4Rd^{-1/4}\right)^{d-1}
    + 1
  \right],
\end{equation}
which completes the proof of the total variation bound.

\paragraph{Lipschitz constant of the gradient.}
Recall that for any any $j \in \{1, \dots,d\}$ we have
\begin{equation}
  \phi^{j}(x) - \frac{1}{2}\|x\|_{2}^{2} = -\lambda\log\left(
    \sum_{l=1}^{n_j}
    \exp\left(\frac{\psi^{j}(y^{j}_{l}) - \frac{\|y^{j}_{l}\|_{2}^{2}}{2} +
    \langle x, y^{j}_{l}\rangle}{\lambda}\right)
    \nu^{j}(y^{j}_{l})
    \right).
\end{equation}
Denote $\widetilde{\phi}^{j}(x) = \phi^{j}(x) - \frac{1}{2}\|x\|_{2}^{2}$.
Fix any $x, x'$ and define $g(t) = \widetilde{\phi}^{j}(x + (x'-x)t)$.
Then, for any $t \in [0,1]$ we have
\begin{align}
  \label{eq:second-derivative-phi-tilde}
  g''(s) =
  -\frac{1}{\lambda}\mathrm{Var}_{L \sim \rho_{t}}\left[
    (Y^{j}(x' - x))_{L}
  \right]
  \geq
  -\frac{1}{\lambda}\|x-x'\|_{2}^{2}m_{j}4R^{2},
\end{align}
where
\begin{align}
  \rho_{t}(l)
  \propto
  \nu(y^{j}_{l})
  \exp\left(\frac{\psi^{j}(y^{j}_{l}) - \frac{\|y^{j}_{l}\|_{2}^{2}}{2} +
  \langle x + t(x' - x), y^{j}_{l}\rangle}{\lambda}\right)
\end{align}
and $Y^{j} \in \mathbb{R}^{d \times m_{j}}$ is the matrix whose $l$-th column
is equal to the vector $y^{j}_{l}$.

Because $\widetilde{\psi}^{j}$ is concave, the bound
\eqref{eq:second-derivative-phi-tilde} shows that
$\phi^{j}$ is $1 + \frac{1}{\lambda}m_{j}4R^{2}$-smooth.

Combining the above with the fact that the convex function
$\mathrm{dist}(x,\mathcal{X})$ has $1$-Lipschitz gradient
\cite[Proposition 12.30]{bauschke2017convex} proves the desired smoothness
bound on the function $V_{\sigma}$.

\paragraph{LSI Constant bound.}
The result follows, for example, by applying the sufficient log-Sobolev
inequality criterion stated in
\cite[Corollary 2.1, Equation (2.3)]{cattiaux2010note}, combined with
the bound \eqref{eq:second-derivative-phi-tilde}.
The exact constant appearing in the log-Sobolev inequality can be traced from
\cite[Equation (3.10)]{cattiaux2010note}.

\end{document}